\pdfoutput=1

\documentclass[12pt,makeidx]{amsart}

\usepackage{amsmath,mathrsfs}
 \usepackage{amsfonts}
 \usepackage{amssymb}

\usepackage[cal=cm]{mathalfa}

\usepackage{enumerate,color}
\usepackage{url}

\usepackage{color}
\usepackage[normalem]{ulem}

\usepackage[foot]{amsaddr}

\definecolor{green}{rgb}{0.01, 0.75, 0.24}
\definecolor{maroon}{rgb}{0.52, 0, 0}
\definecolor{purple}{rgb}{.7,0,1}

\newtheorem{theorem}            {Theorem}[section]

\newtheorem{proposition}        [theorem]{Proposition}

\newtheorem{lemma}              [theorem]{Lemma}
\newtheorem{remark}         [theorem]{Remark}

\newtheorem{example}            [theorem]{Example}

\newcommand{\spann}{\text{span}}

\newcommand{\R}{\mathbb{R}}      
\newcommand{\Z}{\mathbb{Z}}      
\newcommand{\N}{\mathbb{N}}      
\newcommand{\C}{\mathbb{C}}      
\newcommand{\D}{\mathbb{D}}	     
\newcommand{\T}{\mathbb{T}}	     

\usepackage{natbib}

\begin{document}

\title[SZEG\H{O} and Widom theorems for finite codimensional subalgebras]{SZEG\H{O} and Widom theorems for finite codimensional subalgebras of a class of uniform algebras}

\author{Douglas T. Pfeffer}\thanks{Portions of this work appeared in the first named author's PhD dissertation, given at the University of Florida, May 2019}
\address{Department of Mathematics, Berry College, PO Box 495014, Mount Berry GA 30149-5014}
\email{dpfeffer@berry.edu}

\author{Michael T. Jury}\thanks{Second named author partially supported by NSF grant DMS-1900364}
\address{Department of Mathematics, University of Florida, PO Box 118105, Gainesville FL 32611-8105}
\email{mjury@ufl.edu}

\begin{abstract}
  We establish versions of Szeg\H{o}'s distance formula and Widom's theorem on invertibility of (a family of)  Toeplitz operators in a class of finite codimension subalgebras of uniform algebras, obtained by imposing a finite number of linear constraints. Each such algebra is naturally represented on a family of reproducing kernel Hilbert spaces, which play a central role in the proofs.  
\end{abstract}

\maketitle


\vspace{1pc}

\footnotesize \noindent \textbf{Data Availability:} Data sharing not applicable to this article as no data sets were generated or analyzed during the current study.

\normalsize

\section{Introduction} 

Let $\C$ denote the complex plane, $\D = \{z\in \C \ : \ |z|<1\}$ denote the unit disk, and let $\T = \{ z\in \C \ : \ |z|=1 \}$ denote the unit circle in the complex plane (so that $\partial \D = \T$). Let $t$ denote Lebesgue measure on $\T$ and let $L^p = L^p(\T)$ be the $L^p$ spaces on $\T$ with respect to the normalized measure $\textstyle \frac{ \mathop{dt}}{2\pi}$.

Let $H^\infty(\D)$ denote the bounded, analytic functions on $\D$ and let $H^2(\D)$ be the Hardy space of analytic functions on $\D$ with square summable power series coefficients. For $p=2,\infty$, we adopt the standard identification of $H^p(\D)$ with $H^p(\T)$, where $H^p(\T)$ is viewed as the subspace of $L^p(\T)$ containing functions $f$ with vanishing negative Fourier coefficients.

Let $\mathscr{P}^+ = \spann\{ e^{int} \ | \ n\in \N\}$ denote the analytic trigonometric polynomials and let $P^2(\mu)$ denote the $L^2(\mu)$ closure of $\mathscr{P}^+$. With $P^2_0(\mu) = \{p\in P^2(\mu) \ | \ p(0)=0 \}$, the following is a result due to Szeg\H{o}:

\begin{theorem}[Szeg\H{o}'s Theorem (p. 49 in \cite{hoffman}) ] \label{classicszego}
	If $\mu> 0$ is a finite measure on $\T$, then
	\[
	\inf\left\{ \int_{\T} |1-p|^2 \mathop{d\mu}  \ : \ p\in P^2_0(\mu)  \right\} = \exp\left(  \frac{1}{2\pi} \int_0^{2\pi} \log(h) \mathop{dt} \right),
	\]
	where $h$ is the Radon-Nikodym derivative of $\mu$ with respect to Lebesgue measure $t$.
\end{theorem}

This paper will generalize Szeg\H{o}'s theorem under the assumption that the measure $\mu$ is absolutely continuous with respect to Lebesgue measure with a strictly positive, continuous Radon-Nikodym derivative.

Widom's theorem provides a characterization of the invertibility of a Toeplitz operator with symbol $\phi \in L^\infty$ in terms of the distance from $\phi$ to $H^\infty$. Fix $\phi \in L^\infty$ and let $P\colon L^2\to H^2$ be the orthogonal projection onto $H^2$. Define $T_\phi\colon H^2\to H^2$ be $Tf = P\phi f$. Such an operator is called a \textit{Toeplitz operator} (with symbol $\phi$).

\begin{theorem}[Widom's Theorem (Theorem 7.30 in \cite{BanachTech})]\label{classicalwidom}
	Suppose $\phi \in L^\infty$ is unimodular. $T_\phi$ is left-invertible if and only if there $\text{dist}( \phi, H^\infty )<1$.
\end{theorem}

Szeg\H{o} \cite{classicalszego} first established his result in 1920, while Widom \cite{Widom1960} first established his result in 1960. Since then, different versions of both have been established for a variety of settings. Specifically, there are two types of generalizations that we focus on:
\begin{enumerate}[(i)]
    \item A change to the underlying set on which our functions are defined. In this vein, let $\Omega$ be a finite (connected) Riemann surface, and let $\mathcal{A}(\Omega)$ be the algebra of holomorphic functions on $\Omega$. In this setting, we define $H^2$ to be the $L^2$-closure of $\mathcal{A}(\Omega)$ with respect to the representing measure for a point in a nontrivial Gleason part. 
    \item The introduction of finitely many algebraic constraints to our algebra of functions to yield a finite-codimensional subalgebra $A\subseteq \mathcal{A}(\Omega)$. As we will record in Theorem \ref{GamelinIterative}, if we pass to an arbitrary finite-codimensional subalgebra, then $A$ arises via the successive application of finitely many algebraic constraints of the 2-point or Neil type.
\end{enumerate}
\noindent For example, the classic disk algebra $A(\D)$ (functions holomorphic on $\D$ and continuous on $\T$) is yielded when $\Omega = \D$ and $A = A(\D)$ (i.e., no algebraic constraints).

In the direction of (i), there have been a few generalizations established. For Szeg\H{o}'s result, we have the following: In 1965, Sarason \cite{sarasonannulus} established a version for the annulus. In 1967, Ahern and Sarason \cite{AhernSarason} established a version for hypo-Dirichlet algebras. For Widom's result, Abrahamse \cite{AbrahamseToeplitz} established a version for multiply connected domains in 1974.

In the direction of (ii), Balasubramanian, McCullough, and Wijesooriya \cite{NeilAlgebra} established versions of both the Szeg\H{o} and Widom results for the Neil Algebra, a constrained subalgebra of $H^\infty(\D)$:
\[
\mathfrak{A} = \{  f\in H^\infty(\D) \ : \ f'(0)=0 \}.
\]
The following is a brief overview of their results:

Let $S = \{ (\alpha,\beta)\in \C^2 \ : \ |\alpha|^2+|\beta|^2=1 \}$ be the compact unit sphere in $\C^2$. For $(\alpha,\beta) \in S$, define the following Hilbert spaces:
\begin{equation} \label{neilspaces}
H^2_{\alpha,\beta} = \{ f\in H^2(\D) \ : \ f(0)\beta=f'(0)\alpha \}.
\end{equation}
In \cite{DPRS}, it is observed that these Hilbert spaces each carry a representation of $\mathfrak{A}$. They go on to show that each $H^2_{\alpha,\beta}$ is a reproducing kernel Hilbert space with kernel
\[
k^{\alpha,\beta}_w(z) = k^{\alpha,\beta}(z,w) = (\alpha+\beta z)\overline{(\alpha+\beta w)} + \frac{ z^2\overline{w^2}}{ 1-z\overline{w} } 
\]
for $z,w \in \D$. It follows, via the the reproducing property, that
\[
\|k_0^{\alpha,\beta}\|^2 = \langle k_0^{\alpha,\beta},k_0^{\alpha,\beta} \rangle = k_0^{\alpha,\beta} = k^{\alpha,\beta}(0,0) = |\alpha|^2.
\]
Denote by $\mathfrak{A}_0$ those functions in $\mathfrak{A}$ that vanish at 0. The following is a rewording of Theorem 1.3 in \cite{NeilAlgebra}:
\begin{theorem}[Reformulated Szeg\H{o} for $\mathfrak{A}$] \label{reformszegoneil}
	Suppose $\rho>0$ is a continuous function on $\T$. Define constants:
	\[
	C_\rho =\frac{1}{2\pi}\int_0^{2\pi} \log(\rho) \mathop{dt}, \hspace{.15in}  \lambda = \frac{\exp(C_\rho)}{2\pi} \int_0^{2\pi} \rho(t) \exp(-it) \mathop{dt}, \hspace{.15in} \text{and} \hspace{.15in} \sigma = \frac{1}{ \sqrt{1+ |\lambda|^2 }  }(1,\lambda)\in S.
	\]
	Then,
	\[
	\inf\left\{  \frac{1}{2\pi} \int_{\T} |1-p|^2 \rho \mathop{dt} \ : \ p\in \mathfrak{A}_0      \right\} = \exp(C_\rho)\left( \frac{1}{ \|k^\sigma_0\|^2 } \right).
	\]
\end{theorem}
Observe that $\exp(C_\rho)$ is exactly the quantity that is found on the right hand side of Theorem \ref{classicszego}.

We now record the Widom result for the Neil Algebra (Theorem 1.6 in \cite{NeilAlgebra}). For each $(\alpha,\beta)\in S$, let $P_{\alpha,\beta}\colon L^2\to H^2_{\alpha,\beta}$ denote the orthogonal projection. Given $\phi \in L^\infty$, define the operator $T^{\alpha,\beta}_\phi\colon H^2_{\alpha,\beta} \to H^2_{\alpha,\beta}$ by $T^{\alpha,\beta}_\phi f = P_{\alpha,\beta} \phi f$. Such an operator is called the Toeplitz operator with symbol $\phi$ with respect to $(\alpha,\beta)$. Let $\mathfrak{A}^{-1}$ denote the collection of  invertible elements of $\mathfrak{A}$.
\begin{theorem}[Widom for Neil Algebra (Theorem 1.6 in \cite{NeilAlgebra}) ] \label{WidomforNeil}
	Suppose $\phi \in L^\infty$ is unimodular. $T^{\alpha,\beta}_\phi$ is left-invertible for each $(\alpha,\beta)\in S$ if and only if $\text{dist}(\phi, \mathfrak{A})<1$. In particular, $T^{\alpha,\beta}_\phi$ is invertible for each $(\alpha,\beta)\in S$ if and only if $\text{dist}(\phi, \mathfrak{A}^{-1})<1$.
\end{theorem}


In the present work, we generalize these results in both of the directions (i) and (ii) described in the remarks proceeding the statement of Theorem \ref{classicszego}. However, in the interest of clarity, we start by stating the results only for the topological generalization discussed in (i).

When changing one's underlying domain away from the open disk $\D$, topological complications arise. For example, when considering $\D$, we have that $L^2 = H^2 \oplus \overline{H^2_0}$. Passing to the finitely connected case, one finds that $H^2 \oplus \overline{H^2_0}$ is no longer all of $L^2$, rather, there is an additional finite-dimensional defect space $N$ such that $L^2 = H^2 \oplus \overline{H^2_0} \oplus N$. This defect space is tied to the number of holes in the domain and is the complexification of the space of real, regular Borel measures on $\partial \Omega$ that annihilate $\mathcal{A}(\Omega)+\overline{\mathcal{A}(\Omega)}$. Abrahamse, in his consideration of multiply connected domains, analyzed this defect space and made heavy use of a theorem that related the space $\overline{H^2_0} \oplus N$ to the space $\overline{v}^{-1}\overline{H^2}$, where $v$ is a function related to the Green's function for $\Omega$. Abrahamse then used a universal covering space and deck transformations to establish his Widom result. (For details on Abrahamse's work described here, see \cite{AbrahamseToeplitz}.) In our work, we will encounter the $N$-space and Green's function (see Subsection \ref{greensection}). However, to obtain our version of the Widom theorem, we will circumvent the use of a universal covering space by instead appealing to the machinery of Ahern and Sarason for hypo-Dirichlet algebras found in \cite{AhernSarason} (we review the relevant material in Subsection \ref{hDAlg}).

For a compact, Hausdorff space $X$, let $C(X)$ denote the continuous, complex-valued functions on $X$. Recall that a \textit{uniform algebra} $\mathscr{A}$ is a uniformly closed subalgebra of $C(X)$ which contains constants and separates points. When endowed with the sup norm $\|f\| = \sup\{|f(x)| \ : \ x\in X  \}$, it becomes a Banach algebra. A classic example is the disk algebra $A(\overline{\D})$ of functions which are continuous on $\T$ and extend to be analytic over $\D$. Let $M_{\mathscr{A}}$ denote its maximal ideal space. Given $x_0 \in M_{\mathscr{A}}$, let $M_{x_0}$ denote the convex space of representing measures for $x_0$. Finally, let $P_{x_0}$ denote the Gleason part that contains $x_0$. Gamelin \cite{Gamelin1} shows that under the following hypotheses: $M_{x_0}$ is finite dimensional; the measure $dm$ is taken from the relative interior of $M_{x_0}$; all of the representing measures for $x_0$ are mutually absolutely continuous; $P_{x_0}$ contains more than one point, $\mathscr{A}$ can be viewed as an algebra of analytic functions defined on a finite (connected) Riemann surface. From this point forward, we fix $x_0$ and $dm$ as above and let $X$ denote the finite (connected) Riemann surface on which $\mathscr{A}$ is defined. In this manner, we see that $\mathscr{A} = \mathcal{A}(X)$.

Let $H^2$ be the $L^2$ closure of $\mathscr{A}$ and let $H^2_0 = \{ f\in H^2 \ : \ \textstyle \int_{ \partial X } f \mathop{dm} = f(x_0)=0  \}$. 
Let $H^\infty$ be the weak-$*$ closure of $\mathscr{A}$ in $L^\infty$. In the same spirit as Abrahamse, Gamelin showed in \S5 of \cite{Gamelin1} that $L^2(\mathop{dm}) = H^2 \oplus \overline{ H^2_0  } \oplus N$ and $H^\infty = H^2 \cap L^\infty$, where $N$ is a finite dimensional subspace of $L^\infty$ arising from the complexification of a finite dimensional real subspace of $L^\infty$. For every $n\in N$, let $H^2_n$ denote the standard $H^2$ space but endowed with the inner product given by
\[
\langle f,g \rangle_n = \int_{ \partial X } f\overline{g} e^n \mathop{dm}.
\]
Defining the map $\pi_n\colon \mathscr{A}\to \mathcal{B}(H^2_n)\colon f\mapsto M_f$, where $M_f\colon H^2_n\to H^2_n\colon h\mapsto fh$, we have that $\pi_n$ is an isometric homomorphism into the bounded linear operators on $H^2_n$. In short, we say that each $H^2_n$ carries a representation for $\mathscr{A}$. We note that, while it is reasonable to discuss when two $H^2_n$ spaces are unitarily equivalent, we do not need such observations in this paper. In our consideration of all $H^2_n$ spaces, we allow ourselves redundancy. Each $H^2_{n}$ is a reproducing kernel Hilbert space with kernel given by $k^{n}$. Let $k^{n}_{x_0}$ denote said kernel evaluated at $x_0$. With $\mathscr{A}_0 = \{ f\in \mathscr{A} \ : \ \textstyle \int_{ \partial X } f \mathop{dm} = f(x_0)=0  \}$, the following is a Szeg\H{o} result for $\mathscr{A}$: 

\begin{theorem}[Szeg\H{o} Theorem for $\mathscr{A}$] \label{SzegoforunconstrainedA}
	Suppose $\rho > 0$ is a continuous function on $\partial X$. Let $\xi \in H^2, \zeta \in \overline{H^2_0}, n\in N$ be functions such that $\log(\rho) = \xi \oplus \zeta \oplus n \in H^2 \oplus \overline{H^2_0} \oplus N$. With $C_\rho := \textstyle \int_{\partial X} \log(\rho) \mathop{dm}$, it follows that
	\[
	\inf\left\{  \int_{\partial X} |1-p|^2 \rho \mathop{dm} \ : \ p\in  \mathscr{A}_0  \right\} = \exp(C_\rho) \left(  \frac{1}{ \|k^{n}_{x_0}\|^2  }  \right).
	\]
\end{theorem}

As mentioned earlier, the above theorem carries the assumption that the measure $\rho d\mu$ is absolutely continuous with respect to Lebesgue measure with a strictly positive, continuous Radon-Nikodym derivative $\rho$. In this vein, the above theorem is a generalization of Theorem \ref{classicszego} yielded by  only changing the topological structure of the underlying domain. Wermer \cite{wermer} shows that the uniform algebra $\mathscr{A}$, being an algebra of analytic functions defined on a finite (connected) Riemann surface, is a hypo-Dirichlet algebra. Thus, the above Szeg\H{o} result is simply a reformulated special case of Ahern and Sarason's Theorem 10.1 in \cite{AhernSarason}.

To state a Widom result for $\mathscr{A}=\mathcal{A}(X)$, we construct a slightly different family of Hilbert-Hardy spaces that carry representations for $A$. What will be important is that the family is paramterized by a {\em compact} paramater space. Since $\mathscr{A}$ can also be viewed as a hypo-Dirichlet algebra, we can use the machinery developed in \cite{AhernSarason}. Let $\mathscr{A}^{-1}$ denote the invertible elements in $\mathscr{A}$ and let $S_{x_0}$ denote the real linear space of the set of all differences between pairs of measures in $M_{x_0}$. Ahern and Sarason \cite{AhernSarason} record that, as a byproduct of $\mathscr{A}$ being hypo-Dirichlet, no non-zero measure in $S_{x_0}$ annihilates $\log(|\mathscr{A}^{-1}|  )$ and $S_{x_0}$ has finite dimension $\sigma$. Further, it follows that there are $\sigma$ functions $Z_1,\ldots, Z_\sigma$ and $\sigma$ measures $\nu_1,\ldots,\nu_{\sigma}$ in $S_{x_0}$ such that $\textstyle \int_{ \partial X } \log( |Z_j| ) \mathop{d\nu_i} = \delta_{ji}$. For $\alpha=(\alpha_1,\ldots,\alpha_{\sigma})$, define $|Z|^\alpha := |Z_1|^{\alpha_1}\cdots |Z_\sigma|^{\alpha_{\sigma}}$. In Subsection \ref{repforAsection}, we introduce a compact parameter space $\Sigma$ such that, given $\alpha \in \Sigma$, the spaces $H^2_\alpha$ -- defined to be the usual $H^2$ space but endowed with the inner product given by
\[
\langle f,g \rangle_\alpha = \int_{ \partial X } f\overline{g} |Z|^\alpha dm
\]
-- each carries a representation for $\mathscr{A}$. As with the $H^2_n$ spaces, this representation is witnessed by the isometric homomorphism $\pi_\alpha\colon \mathscr{A}\to \mathcal{B}(H^2_\alpha)\colon f\mapsto M_f$, where $M_f\colon H^2_\alpha \to H^2_\alpha\colon h\mapsto fh$. In this case, however, our construction of the parameter space $\Sigma$ will involve passing to a quotient space. As such, questions about the unitary equivalence of two $H^2_\alpha$ spaces is more relevant. Proposition \ref{equivalent} in Subsection \ref{repforAsection} establishes that two tuples $\alpha_1$ and $\alpha_2$ in $\Sigma$ belong to the same equivalence class in $\Sigma$ if and only if $H^2_{\alpha_1}$ and $H^2_{\alpha_2}$ are unitarily equivalent.

For $\phi \in L^\infty$, let $M_\phi$ denote the operator that multiplies by $\phi$. For $\alpha \in \Sigma$, let $V_{\alpha}\colon H^2_{\alpha}\to L^2_\alpha$ be the inclusion map. Thus, $P_{\alpha} = V_{\alpha}V^*_{\alpha}\colon L^2_{\alpha} \to H^2_{\alpha,D}$ is the orthogonal projection onto $H^2_{\alpha}$. For a fixed $\phi \in L^\infty$, we define $T^{\alpha}_\phi\colon H^2_{\alpha}\to H^2_{\alpha}$ by
\[
T^{\alpha}_{\phi} =P_{\alpha}M_\phi =  V^*_{\alpha}M_{\phi} V_{\alpha}.
\]
Call $T^{\alpha}_\phi$ the \textit{Toeplitz operator with symbol $\phi$ with respect to $\alpha$   }. The following is the Widom theorem for $\mathscr{A}$:

\vspace{.1in}

\begin{theorem}[Widom Theorem for $\mathscr{A}$] \label{widomforunconstrainedA}
	Suppose $\phi \in L^\infty$ is unimodular. $T^{\alpha}_\phi$ is left-invertible for each $\alpha \in \Sigma$ if and only if $\text{dist}(\phi, \mathscr{A}) <1$. In particular, $T^{\alpha}_\phi$ is invertible for each $\alpha\in \Sigma$ if and only if $\text{dist}(\phi,\mathscr{A}^{-1})<1$.
\end{theorem}

When $X=\Omega$ is a multiply connected planar domain, the above theorem becomes Abrahamse's Theorem 4.1 in \cite{AbrahamseToeplitz}.

Theorems \ref{SzegoforunconstrainedA} and \ref{widomforunconstrainedA} are generalizations of the Szeg\H{o} and Widom theorems when the underlying domain is changed to a finite (connected) Riemann surface. The other type of generalization is obtained by passing to a finite codimension subalgebra, obtained by imposing a finite number of linear constraints. The prototypical example for this is the Neil Algebra $\mathfrak{A}$. For the Neil Algebra (and therefore Theorems \ref{reformszegoneil} and \ref{WidomforNeil}) the underlying domain is the disk $X = \D$.

To formulate the Szeg\H{o} and Widom theorems in the constrained case, we again let $\mathscr{A} = \mathcal{A}(X)$ and now let $A\subseteq \mathscr{A}$ be a finite codimensional subalgebra with codimension $d$. Let $\Delta := \textstyle \prod_{1}^d (  \C \cup \{\infty\} )$ and denote its elements by $D = ( t_1,\ldots, t_d )$. A theorem due to Gamelin (reproduced in this paper as Theorem \ref{GamelinIterative}) details explicitly how $A$ is constructed from $\mathscr{A}$ via inductively imposing algebraic constrains.

Within each $H^2_n$, we develop a family of Hilbert-Hardy spaces that carry representations for $A$. This family, denoted $H^2_{n,D}$ with $(n,D)\in N\times \Delta$, are constructed iteratively from $H^2_n$ by encoding the algebraic constraints that built $A$ from $\mathscr{A}$ (see Subsection \ref{repforA} for their exact construction). Each $H^2_{n,D}$ is a reproducing kernel Hilbert space with kernel given by $k^{n,D}$. Let $k^{n,D}_{x_0}$ denote said kernel evaluated at $x_0$. With $A_0 = \{ f\in A \ : \ \textstyle \int_{ \partial X } f \mathop{dm} = f(x_0)=0  \}$, the following is our Szeg\H{o} result: 

\begin{theorem}[Szeg\H{o} Theorem for $A$] \label{SzegoforA}
	Suppose $\rho > 0$ is a continuous function on $\partial X$. Let $\xi \in H^2, \zeta \in \overline{H^2_0}, n\in N$ be functions such that $\log(\rho) = \xi \oplus \zeta \oplus n \in H^2 \oplus \overline{H^2_0} \oplus N$. Let $D \in \Delta$ be the unique tuple such that $e^\xi \in H^2_{n,D}$. With $C_\rho := \textstyle \int_{\partial X} \log(\rho) \mathop{dm}$, it follows that
	\[
	\inf\left\{  \int_{\partial X} |1-p|^2 \rho \mathop{dm} \ : \ p\in  A_0  \right\} = \exp(C_\rho) \left(  \frac{1}{ \|k^{n,D}_{x_0}\|^2  }  \right).
	\]
\end{theorem}

To state our Widom result, we again use the $H^2_\alpha$ spaces instead. Within each $H^2_\alpha$, we define a family of Hilbert-Hardy spaces $H^2_{\alpha,D}$ with $(\alpha,D) \in \Sigma\times \Delta$ such that each $H^2_{\alpha,D}$ carries a representation for $A$.

For $\phi \in L^\infty$, let $M_\phi$ denote the operator that multiplies by $\phi$. For $(\alpha,D) \in \Sigma \times \Delta$, let $V_{\alpha,D}\colon H^2_{\alpha,D}\to L^2_\alpha$ be the inclusion map. Thus, $P_{\alpha,D} = V_{\alpha,D}V^*_{\alpha,D}\colon L^2_{\alpha} \to H^2_{\alpha,D}$ is the orthogonal projection onto $H^2_{\alpha,D}$. For a fixed $\phi \in L^\infty$, we define $T^{\alpha,D}_\phi\colon H^2_{\alpha,D}\to H^2_{\alpha,D}$ by
\[
T^{\alpha,D}_{\phi} =P_{\alpha,D}M_\phi =  V^*_{\alpha,D}M_{\phi} V_{\alpha,D}.
\]
Call $T^{\alpha,D}_\phi$ the \textit{Toeplitz operator with symbol $\phi$ with respect to $(\alpha,D)$   }. Let $A^{-1}$ denote the collection of invertible elements of $A$. The following is the Widom theorem for $A$:

\vspace{.1in}

\begin{theorem}[Widom Theorem for $A$] \label{widomforA}
	Suppose $\phi \in L^\infty$ is unimodular. $T^{\alpha,D}_\phi$ is left-invertible for each $(\alpha,D) \in \Sigma \times \Delta$ if and only if $\text{dist}(\phi, A) <1$. In particular, $T^{\alpha,D}_\phi$ is invertible for each $(\alpha,D)\in \Sigma \times \Delta$ if and only if $\text{dist}(\phi,A^{-1})<1$.
\end{theorem}

The remainder of the paper is devoted the proofs of Theorems \ref{SzegoforA} and \ref{widomforA}. These theorems encode both generalizations due to a change in the underlying domain's topology, as well as those due the introduction of algebraic constraints. Some problems of this sort have been considered previously in the literature. In the case of the Neil algebra, a Pick-interpolation result has been established in \cite{DPRS} and an investigation into the spectrum of its Toeplitz operators has been carried out in \cite{Broschinski}. More generally, for results related to constrained algebras, see \cite{Ballconstrained}, \cite{dritschelconstrained}, \cite{Raghupathi2point}, \cite{Trent1}, \cite{Trent2}, and \cite{Banjade}. In particular, in the special case when the underlying domain is the disk, a Widom-type invertibility theorem for familiies of Toeplitz operators in the constrained case was obtained by Anderson and Rochberg \cite{Anderson-ROchberg}. For results on multiply connected domains, see \cite{Ballconstrained} and \cite{Mult-ConnDom}. We also note that there has been work on a Szeg\H{o} theorem in noncommutative settings. Specifically, where one considers Arveson subdiagonal algebras inside a finite von Neumann algebra. This setting generalizes $H^\infty(\D)$ to a non-commutative $H^\infty$. For work in this direction, see \cite{noncomSzego1} and \cite{noncomSzego2}.

\subsection{Reader's Guide} In Section~\ref{sec:setup} we collect some preliminary material on hypo-Dirichlet algebras and their constrained subalgebras. Of particular importance will be the results of Ahern and Sarason, \cite{AhernSarason} and \cite{AhSa2}, on hypo-Dirichlet algebras and some results of Gamelin \cite{Gamelin1} on the structure of constrained subalgebras. In both cases we obtain families of reproducing kernel Hilbert spaces on the underlying domain, paramaterized in a suitable way. In Section~\ref{lemmaforszego} we prove Theorem~\ref{SzegoforA}. Section~\ref{lemmassection} contains some additional preliminary material on famillies of Toeplitz operators, and finally Theorem~\ref{widomforA} is proved in Section~\ref{sec:widom}.

\section{Setup} \label{sec:setup}

\subsection{Finite-Codimensional Subalgebras $A$ of $\mathscr{A}$} \label{FCS}

In this section we review Gamelin's characterization of finite codimension subalgebras, and fix some facts and notation that will be used in the sequel.

Let $\mathscr{A}$ be a uniform algebra defined on $X$. Let $x_0\in X$ and $\mathop{dm}\in M_{x_0}$ such that $X$ is a finite (connected) Riemann surface.

Given a point $\theta \in X$, a \textit{point derivation} at $\theta$ is a linear functional $D_\theta$ on $\mathscr{A}$ which satisfies
\[
D_\theta(fg) = f(\theta) D_\theta(g) + g(\theta) D_\theta(f).
\]
A subalgebra $B\subseteq \mathscr{A}$ is a $\theta$-\textit{subalgebra} if there is a sequence of subalgebras $\mathscr{A}=A_0\supseteq A_1 \supseteq \ldots \supseteq A_k = B$ such that $A_i$ is the kernel of a continuous point derivation $D_i$ of $A_{i-1}$ at $\theta$.  The following is an explicit description of all finite codimensional subalgebras $A$ of $\mathscr{A}$:
\begin{theorem}[Theorem 9.8 in \cite{Gamelin1}] \label{GamelinIterative}
	 If $A\subseteq \mathscr{A}$ is a finite codimensional subalgebra, then $A$ can be obtained from $\mathscr{A}$ in two steps:
	\begin{enumerate}[(i)]
		\item There exists a finite number $\ell$ and, for $1\leq i \leq \ell$, pairs of points $a_i, b_i \in X$  such that if
		\[
		B := \{  f\in \mathscr{A} \ : \ f(a_i) = f(b_i) \ \text{for all} \ 1\leq i \leq \ell \},
		\]
		then $A \subseteq B \subseteq \mathscr{A}$.
		\item There exists a finite number $k$, and, for $1\leq j \leq k$, distinct points $c_j\in X$ and $c_j$-algebras $B_j$ of $B$ such that
		then $A = B_1\cap \ldots \cap B_k$.
	\end{enumerate}
\end{theorem}

One may interpret the construction in the following way: All finite codimensional subalgebras $A$ of $\mathscr{A}$ are obtained by iteratively imposing a finite number of algebraic constraints. In particular, there exists a chain $A = A_d \subseteq A_{d-1} \subseteq \ldots \subseteq A_1 \subseteq A_0 = \mathscr{A}$ such that at the $i^{\text{th} }$ step
\begin{enumerate}[(i)]
	\item $A_i = \{ f\in A_{i-1} \ : \ f(a)=f(b) \}$ for some $a,b \in X$ or,\label{2pointuniform}
	\item $A_i$ is the kernel of a continuous point derivation of $A_{i-1}$ at some point $c\in X$. \label{neiluniform}.
\end{enumerate} 
We will refer to the first constraint as \textit{2-point constraint} and the second as a \textit{Neil constraint}.

In this manner each $A_i$ is a codimension one sublagebra of $A_{i-1}$ and $d$ is the codimension of $A$ in $\mathscr{A}$. The chain of subalgebras $A = A_d \subseteq A_{d-1} \subseteq \ldots \subseteq A_1 \subseteq A_0 = \mathscr{A}$ is called a \textit{Gamelin chain}. Let $\Gamma$ denote the set of points in $X$ that the algebraic constraints are defined on. Let $\gamma$ denote the total number of constrained values in the creation of $A$. Thus, given a function $f\in A$, we let $f_\Gamma \in \C^\gamma$ be the vector whose entries consist of $f$ either evaluated at various points or its derivatives evaluated at various points (depending on how the points are encoded into the construction of $A$).

\begin{example}
	Given a uniform algebra $A_0 := \mathscr{A}$,
	\[
	A = \{ f\in \mathscr{A} \ : \ f(a)=f(b) \ \text{and} \ f'(c)=f'''(c)=0 \}
	\]
	is a finite codimensional subalgebra. We can construct it in the following way: Construct $A_1 = \{f\in \mathscr{A} \ : \ f(a)=f(b) \} \subseteq \mathscr{A}$. The functional $D'_c\colon A_1\to \C\colon f\mapsto f'(c)$ defines a continuous point derivation of $A_1$ at $c$. Put $A_2 = \ker(D'_c) = \{ f\in A_1 \ : \ f'(c)=0\}.$ Now consider the functional $D'''_c\colon A_2 \to \C: f\mapsto f'''(c)$. Observe that, given $f,g \in A_2$, we have that $f'(c)=g'(c)=0$ and thus
	\[
	(fg)'''(c) = f'''(c)g(c) + 3f''(c)g'(c) + 3f'(c)g''(c) + f(c)g'''(c) = f'''(c)g(c) + f(c)g'''(c).
	\]
	Therefore $D'''_c$ defines a continuous point derivation of $A_2$ at $c$. With $A_3 = \ker(D'''_c) = \{f\in A_2 \ : \ f'''(c) = 0  \}$, it follows that $A = A_3$. Further, we have $\Gamma = \{a,b,c\}$, $\gamma = 4$ and, given $f\in A$, $\Gamma_f = (f(a), f(b), f'(c), f'''(c))^\top \in \C^4$.
\end{example}

\subsection{Hypo-Dirichlet Algebras} \label{hDAlg}

In \cite{wermer}, Wermer showed that algebras defined on finite (connected) Riemann surfaces are hypo-Dirichlet. In \cite{AhernSarason}, Ahern and Sarason investigated these algebras in further detail. In this subsection, we reproduce the parts of their work that we'll use frequently.

Given our uniform algebra $\mathscr{A}$, let $\mathscr{A}^{-1}$ denote the collection of its invertible elements. Now, $\mathscr{A}$ being a hypo-Dirichlet algebra over $X$ guarantees the following:
\begin{enumerate}[(I)]
	\item The real linear span of $\log(| \mathscr{A}^{-1}| )$ is uniformly dense in $C_{\R}(X)$ (the space of real, continuous functions on $X$); \label{hDglobal1}
	\item The uniform closure of $\text{Re}(\mathscr{A})$ has finite codimension in $C_{\R}(X)$.\label{hDglobal2}
\end{enumerate}

Algebras that obey property (\ref{hDglobal1}) are referred to as \textit{logmodular algebras} (see, Section II.4 of \cite{GamelinBook}). It is the additional property (\ref{hDglobal2}) that distinguishes hypo-Dirichlet algebras. Letting $S_{x_0}$ denote the real linear span of the set of all differences between pairs of measures in $M_{x_0}$, we observe that conditions (\ref{hDglobal1}) and (\ref{hDglobal2}) above imply the following local variants:
\begin{enumerate}[(I$'$)]
	\item No non-zero measure in $S_{x_0}$ annihilates $\log(|\mathscr{A}^{-1}|  )$;
	\item $S_{x_0}$ has finite dimension $\sigma$.
\end{enumerate}
 By (II$'$), we can put $S_{x_0} = \spann_{\R}\{  \mu_1,\ldots, \mu_\sigma \}$. Corollary 1 in \S 3 in \cite{AhernSarason} shows that each $\mu_i$ is absolutely continuous with respect to $\mathop{dm}$. Now, put $\lambda_i := \mathop{d\mu_i}/\mathop{dm}$ and define
\[
N:= \spann_{\C}\{ \lambda_1,\ldots,\lambda_\sigma \}.
\]
This $N$-space turns out to be the same space that was mentioned in the Introduction. Specifically, it is the same space that Gamelin discussed in \S5 of \cite{Gamelin1}. Details on this space and its relation to algebras defined on multiply connected domains can be found in Section 4.5 of \cite{Fisher}. We reproduce the necessary information in Subsection \ref{greensection}.

It also follows from (\ref{hDglobal1}) and (\ref{hDglobal2}) that there are $\sigma$ functions $Z_1,\ldots, Z_{\sigma}$ in $\mathscr{A}^{-1}$ and $\sigma$ measures $\nu_1,\ldots,\nu_\sigma$ in $S_{x_0}$ such that
\begin{equation} \label{basis}
\int_{\partial X} \log( |Z_j| ) \mathop{d\nu_i} = \delta_{ji}.
\end{equation}
We will fix such functions and measures. A small note on notation: For $\alpha = (\alpha_1,\ldots,\alpha_{\sigma}) \in \R^\sigma$, we define
\[
|Z|^\alpha = |Z_1|^{\alpha_1}\cdots |Z_\sigma |^{\alpha_\sigma}.
\]
Recalling that $H^\infty$ is the weak-* closure of $\mathscr{A}$ in $L^\infty$, the following lemma is an essential part of the investigations carried out by Ahern and Sarason:
\begin{lemma}[Lemma 10.1 in \cite{AhernSarason}] \label{ahernlemma}
    Let $\alpha$ be a $\sigma$-tuple in $\R^\sigma$. Then there is a function $h\in H^\infty$ such that $|h|=|Z|^\alpha$ almost everywhere.
\end{lemma}
Borrowing from Ahern and Sarason, we will refer to a function $h \in H^p$ as an \textit{inner function} if there exists an $\alpha \in \Sigma$ such that $|h| = |Z|^\alpha$. A function $g \in H^p$ is an \textit{outer function} if  $\textstyle \log(| \int_{\partial X} g \mathop{dm}|  ) =  \int_{\partial X } \log(|g|)\mathop{dm} > -\infty$. These inner functions contain zeros inside of $X$ but, unlike the in the disk, are not unimodular on the boundary; however they do act as isometric multipliers between $L^p_\alpha$ spaces for different $\alpha$'s.

For $f$, a function on $X$, we let $\textstyle \int f \mathop{d\nu}$ denote the $\sigma$-tuple
\[
\left(  \int_{\partial X} f \mathop{d\nu_1},\ldots, \int_{\partial X} f \mathop{d\nu_{\sigma}}   \right)
\]
(provided each of the individual integrals exist). We then have the Ahern-Sarason inner-outer factorization:
\begin{theorem}[Theorem 7.2 in \cite{AhernSarason}] \label{ahernfactorization}
     Let $f$ be a function in $H^p$ $(1\leq p< \infty$ such that $|f|$ is log-integrable with respect to all representing measures in $M_{x_0}$. Then there are, in $H^p$, an outer function $g$ and an inner function $h$ such that $f=gh$ and $\textstyle \int_{\partial X} \log(|g|) \mathop{d\nu} = (0,\ldots,0)$. The functions $g$ and $h$ are uniquely determined by $f$ to within multiplicative constants of unit modulus.
\end{theorem}

\subsection{Representations for $\mathscr{A}$} \label{repforAsection}

 With the notation inherited from the previous subsection, let
\[
\mathcal{L} := \left\{  \int_{\partial X} \log(|h|) \mathop{d\nu} \ : \ h\in \mathscr{A}^{-1}  \right\} \subseteq \R^{\sigma}.
\]
Observe that, since each of the $Z_j$ are in $\mathscr{A}^{-1}$, (\ref{basis}) shows that $\mathcal{L}$ contains the standard basis vectors $e_j = (0,\ldots,0,1,0,\ldots,0)$ where the 1 occurs in the $j^{\text{th}}$ entry. Thus $\mathcal{L}$ is at least a $\sigma$-dimensional subgroup of $\R^\sigma$. Theorem 8.1 in \cite{AhernSarason} shows that $\mathcal{L}$ is discrete as well. Thus, not only is $\mathcal{L}$ isomorphic to $\Z^\sigma$, but the quotient $\R^\sigma / \mathcal{L}$ is isomorphic to the $\sigma$-torus $\T^\sigma$. In particular, this quotient is compact. We will let $\Sigma$ denote $\R^\sigma / \mathcal{L}$.

Given any $\alpha$ taken from any equivalence class in $[\alpha]\in\Sigma$, let $H^2_\alpha$ be the usual $H^2$ space but endowed with the following inner product:
\begin{equation} \label{H2alphainnerproduct}
\langle f,g \rangle_\alpha = \int_{\partial X} f\overline{g} |Z|^{\alpha} \mathop{dm}.
\end{equation}
As mentioned in the Introduction, each of these spaces carry a representation for $\mathscr{A}$. The following proposition establishes when two $H^2_\alpha$ spaces are unitarily equivalent:

\begin{proposition} \label{equivalent}
	Given two $\sigma$-tuples $\alpha_1$ and $\alpha_2$, they both belong to the same equivalence class in $\Sigma$ if and only if $H^2_{\alpha_1}$ and $H^2_{\alpha_2}$ are unitarily equivalent.
\end{proposition}
\begin{proof} To start, suppose $\alpha_1, \alpha_2 \in [\alpha] \in \Sigma = \R^\sigma / \mathcal{L}$. Then there exists $\ell \in \mathcal{L}$ such that $\alpha_1 = \alpha_2 + \ell$. In particular, $|Z|^{\alpha_1} = |Z|^{\alpha_2 + \ell}$ so that $|Z|^\ell = |Z|^{\alpha_1-\alpha_2}$.
	
	Since the function $|Z|^\ell$ is non-negative and in $L^1$, it follows from Theorem 6.1 in \cite{AhernSarason} that there exists an outer function $h$ in $H^1$ such that $|h| = |Z|^\ell$ almost everywhere. Since $h$ is an outer function, it has no zeros inside $X$. The fact that $|h| = |Z|^\ell$ guarantees that $h$ has no zeros on $\partial X$ as well. Thus $h$ is invertible in $A$ such that $|h| = |Z|^\ell = |Z|^{\alpha_1-\alpha_2}$. Thus $|Z|^{\alpha_1} = |h||Z|^{\alpha_2}$. It follows that the  $H^2_{\alpha_1}$ and $H^2_{\alpha_2}$ are unitarily equivalent -- witnessed by the multiplication operator $M_{ |h|^{1/2 } }$.
	
	Conversely, suppose $H^2_{\alpha_1}$ and $H^2_{\alpha_2}$ are unitarily equivalent for $\sigma$-tuples $\alpha_1$ and $\alpha_2$. Then there exists a unitary operator $U$ such that, for all functions $\phi \in H^2$, $U M^1_\phi = M^2_\phi$, where $M^i_\phi$ is the operator on $H^2_{\alpha_i}$ that multiplies by $\phi$.
	
	Now, let $k_w^{i}(z)$ be the reproducing kernel for $H^2_{\alpha_i}$. Observe that if $f\in H^2_{\alpha_i}$, then
	\[
	\langle f, (M^i_\phi)^* k_w^{i} \rangle_2 = \langle M^i_\phi f, k_w^{i}\rangle_2 = \langle \phi f, k_w^{i}\rangle_2 = \phi(w) \langle f, k_w^{i} \rangle_2 = \langle f, \overline{\phi(w)} k_w^{i} \rangle_2.
	\]
	Thus we yield the following eigenvector relationships:
	\[
	(M^1_\phi)^* k_w^{1} = \overline{\phi(w) } k_w^{1} \hspace{.25in} \text{and} \hspace{.25in} (M^2_\phi)^* k_w^{1} = \overline{\phi(w) } k_w^{2}.
	\]
	These relationships immediately imply that $\ker( (M_\phi^1)^* - \overline{\phi(w)}I  ) = \C k_w^1$ and $\ker( (M_\phi^2)^* - \overline{\phi(w)}I  ) = \C k_w^2$. Since unitary maps map kernel spaces to one another, we must have $U k_w^1 = f(w) k_w^2$, where $f(w)$ is a scalar valued function in $\mathscr{A}$ dependent only on $w$. In a reproducing kernel Hilbert space, it suffices to show equality on the kernels, therefore $U = M_f$. We also have that $U^{-1} = M_{ f^{-1} }$. Therefore the unitary operator $U$ is given by multiplication by the invertible function $f$.
	
	It follows that $|Z|^{\alpha_1}=|f| |Z|^{\alpha_2}$. Taking logs and integrating both sides shows that $\ell_f := \textstyle \int_{\partial X} \log(|f|) \mathop{d\nu} = \alpha_1-\alpha_2$. Thus, $\alpha_1$ and $\alpha_2$ differ by the coordinate of an invertible element of $A$ -- meaning $\alpha_1 - \alpha_2 \in \mathcal{L}$. This puts $\alpha_1$ and $\alpha_2$ in the same equivalence class in $\Sigma$. 
\end{proof}

\begin{remark} \label{Nspace}
	In light of Proposition \ref{equivalent}, we will denote by $\alpha$ the corresponding equivalence class $[\alpha]\in \Sigma = \R^\sigma / \mathcal{L}$.
\end{remark}

As already mentioned in the Introduction, there is another way to construct representation-carrying spaces for $\mathscr{A}$ (albeit, in a manner that does not produce a compact space of parameters). For $n\in N$, let $H^2_n$ denote the standard $H^2$ space but with the inner product defined by
\[
\langle f, g \rangle_n = \int_X f\overline{g} \mathop{e^ndm}.
\]
Each $H^2_n$ defines a reproducing kernel Hilbert space and carries a representation for $\mathscr{A}$. 

\subsection{Representations for $A$} \label{repforA}

Let $A$ be a finite codimensional subalgebra of $\mathscr{A}$ generated via the Gamelin chain $A = A_d \subseteq A_{d-1}\subseteq \ldots \subseteq A_1 \subseteq A_0 = \mathscr{A}$. While the representations $\pi_n\colon \mathscr{A} \to \mathcal{B}(H^2_n)\colon f\mapsto M_f$ and $\pi_\alpha\colon \mathscr{A} \to \mathcal{B}(H^2_\alpha)$ clearly give representations of $\mathscr{A}$, it is less obvious how to construct spaces $H^2_{n,D} \subseteq H^2_n$ and $H^2_{\alpha,D} \subseteq H^2_\alpha$ which are invariant under $M_f$ for $f\in A\subseteq \mathscr{A}$ (but not necessarily invariant for $f\in \mathscr{A}$) and thereby generate a richer class of representations $\pi_{n,D}\colon f\mapsto M_f\big|_{H^2_{n,D}}$ and $\pi_{\alpha,D}\colon f\mapsto M_f\Big|_{H^2_{\alpha,D}}$ for $f\in A$, the subalgebra of $\mathscr{A}$. 
We take care of this issue next. 

This construction is formally the same whether  we work inside $H^2_n$ or $H^2_\alpha$; therefore in describing the construction we temporarily write $H^2$ to mean either $H^2_n$ or $H^2_\alpha$ and use the notation $H^2_D$ to mean either $H^2_{n,D}$ or $H^2_{\alpha,D}$ depending on the choice of meaning for the notation $H^2$.


The representations will be built inductively via the Gamelin chain. $A_1$ can be constructed from $A_0=\mathscr{A}$ in one of two ways:
\begin{enumerate}[(i)]
	\item $A_1 = \{ f\in A_0 \ : \ f(a) = f(b) \}$ for some $a,b \in X$ or,
	\item $A_1 = \{ f\in A_0 \ : \ f'(c)=0 \}$. for some $c \in X$
\end{enumerate} 

If (i) occurs, then we form 
\[
H^2_{t_1} = \{ f\in H^2 \ : \ f(a)=t_1f(b) \} = \{k_a^0 - t_1 k_b^0 \}^\perp \subseteq H^2
\]
where $t_1 \in \C\cup\{\infty\}$, and $k_a^0$ and $k_b^0$ are the reproducing kernels in $H^2$ at $a$ and $b$ respectively. Observe that $H^2_{t_1}$ is invariant for $A_1$ and hence $H^2_{t_1}$ carries a representation for $A_1$

If (ii) occurs, then we form 
\[
H^2_{t_1} = \{ f\in H^2 \ : \ f(c)=t_1f'(c) \} = \{ k_c^0 - t_1 k^0_{c^{(1)}} \}^\perp \subseteq H^2,
\]
where $t_1 \in \C\cup \{\infty\}$, $k_c^0$ is the reproducing kernel in $H^2$ at $c$, and $k^0_{c^(1)}$ is the reproducing function in $H^2$ that returns a function's first derivative at $c$. It follows from the Liebniz rule that $H^2_{t_1}$ is invariant for $A_1$ and hence $H^2_{t_1}$ carries a representation for $A_1$.

Proceeding in this manner along the Gamelin chain, we assume that $H^2_{t_{i-1}}$ holds a representation for $A_{i-1}$. $A_i$ can only be built from $A_{i-1}$ in one of two ways: 

\begin{enumerate}[(i)]
	\item $A_i = \{ f\in A_{i-1} \ : \ f(a)=f(b) \}$ for some $a,b \in X$ or,
	\item $A_i$ is the kernel of a continuous point derivation $D_c$ of $A_{i-1}$ at some point $c\in X$.
\end{enumerate} 

If (i) occurs, then we form 
\[
H^2_{t_i} = \{ f\in H^2_{t_{i-1}} \ : \ f(a)=t_if(b) \} = \{k_a^{i-1} - t_i k_b^{i-1} \}^\perp \subseteq H^2_{t_{i-1}},
\]
where $t_i \in \C\cup\{\infty\}$, and $k_a^{i-1}$ and $k_b^{i-1}$ are the reproducing kernels in $H^2_{t_{i-1}}$ at $a$ and $b$ respectively. We claim that $H^2_{t_{i}}$ is invariant for $A_i$. Since $A_i \subseteq A_{i-1}$ and $A_{i-1}H^2_{t_{i-1}} \subseteq H^2_{t_{i-1}}$, it follows that $A_i H^2_{t_{i-1}}\subseteq H^2_{t_{i-1}}$. Finally, given $g\in A_i$ and $f\in H^2_{t_i}$, we have that $(fg)(a) = f(a)g(a) = t_i f(b) g(b)$ and thus $fg\in H^2_{t_i}$ This shows that $H^2_{t_i}$ is invariant for $A_i$ hence $H^2_{t_i}$ carries a representation for $A_i$.

If (ii) occurs at the $i^{\text{th}}$ iteration. In this case, there exists a natural number $n$ such that $A_i = \ker(D_c) = \{  f\in A_{i-1} \ : \ D_c(f)=f^{(n)}(c) = 0  \}$. Form 
\[
H^2_{t_i} = \{ f\in H^2_{t_{i-1}} \ : \ f(c)=t_if^{(n)}(c) \} = \{ k_c^{i-1} - t_i k^{i-1}_{c^{(n)}} \}^\perp \subseteq H^2_{t_{i-1}},
\]
where $t_i \in \C\cup \{\infty\}$, $k_c^{i-1}$ is the reproducing kernel in $H^2_{t_{i-1}}$ at $c$, and $k^{i-1}_{c^{(n)}}$ is the reproducing function in $H^2_{t_{i-1}}$ that returns a function's $n^{\text{th}}$ derivative at $c$. We claim that this Hilbert space is invariant for $A_i$.

As before, we  know $A_i H^2_{t_{i-1}}\subseteq H^2_{t_{i-1}}$. We need only show that, if $g\in A_i$ and $f\in H^2_{t_{i}}$, then $(fg)(c) =t_i (fg)^{(n)}(c)$. To see this, first take $f\in M:=\{ f\in A_{i} \ : \ f(c)=t_i f^{(n)}(c)  \} \subseteq H^2_{t_i}$.  Observe that, since $fg\in A_{i-1}$, the fact that $D_c$ is a continuous point derivation shows that
\begin{equation} \label{pointderiv1}
t_i (fg)^{(n)}(c) = t_i D_c(fg) = t_i( D_c(f)g(c) + D_c(g)f(c) ) = t_i( f^{(n)}(c)g(c) + g^{(n)}(c)f(c) ).
\end{equation}
However, since $g\in A_i$ and $f\in M$,
\begin{equation} \label{pointderiv2}
t_i( f^{(n)}(c)g(c) + g^{(n)}(c)f(c) ) = t_i f^{(n)}(c) g(c) = f(c)g(c) = (fg)(c).
\end{equation}
(\ref{pointderiv1}) and (\ref{pointderiv2}) show that $(fg)(c) =t_i (fg)^{(n)}(c)$. However, since $M$ is dense in $H^2_{t_i}$, it follows that if $g\in A_i$ and $f\in H^2_{t_{i}}$, then $fg\in H^2_{t_i}$ -- guaranteeing that $H^2_{t_i}$ is invariant for $A_i$ and hence $H^2_{t_i}$ carries a representation for $A_i$.

Thus, by induction, we have built a reproducing kernel Hilbert space $H^2_{t_d}$ that carries a representation for $A = A_d$. In particular, it is constructed by building a chain of Hilbert spaces
\[
H^2_{t_d} \subseteq H^2_{t_{d-1}} \subseteq \ldots \subseteq H^2_{t_1}\subseteq H^2_{t_0} = H^2
\]
where each $H^2_i$ is a codimension-1 subspace of $H^2_{i-1}$ and each $H^2_{t_i}$ carries a representation for $A_i$. Thus, our representations for $A$ are parametrized by the $d$-tuple $(t_1,\ldots,t_d) \in \prod_{1}^d (\C\cup \{\infty\})$. 

As mentioned in the Introduction, we will denote the compact product $\prod_{1}^d (\C\cup \{\infty\})$ by $\Delta$ and its tuples by $D$. Further, given a tuple $D = (t_1,\ldots,t_d)$, we will instead denote by $H^2_D$ the space $H^2_{t_d}$ that carries the representation for $A$.

We now introduce a multiplication on $\Delta$ which makes $\Delta$ an algebra. Let $D = (t_1,\ldots,t_d)$ and $\widetilde{D} = (s_1,\ldots,s_d)$ both be tuples in $\Delta$. Let $f\in H^2_D$ and $g \in H^2_{\widetilde{D}}$. It follows that the product $fg$ will belong to $H^2_{\widehat{D}}$ where $\widehat{D}=(r_1,\ldots,r_d)$ is defined as follows:
\[
(\widehat{D})_i := r_i = \begin{cases}
t_is_i &\text{if the} \ i^{\text{th}} \ \text{constraint in the Gamelin chain is a 2-point constraint}\\
\frac{ 1}{ \frac{1}{t_i} + \frac{1}{s_i}  }   & \text{if the} \ i^{\text{th}} \ \text{constraint in the Gamelin chain is a Neil constraint}
\end{cases}
\]
In this manner, $\widehat{D}$ is uniquely defined. 
Next, given a function $f\in H^2_{D}$ where $D = (t_1,\ldots, t_d)$, we have that (provided it exists) $f^{-1} \in H^2_{D^{-1}}$ where $D^{-1}$ is given by
\[
(D^{-1})_i = \begin{cases}
\frac{1}{t_i} &\text{if the} \ i^{\text{th}} \ \text{constraint in the Gamelin chain is a 2-point constraint}\\
-t_i & \text{if the} \ i^{\text{th}} \ \text{constraint in the Gamelin chain is a Neil constraint}
\end{cases}
\]
Lastly, we will denote by $D_\Gamma$ the $d$-tuple defined by: 
\[
(D_\Gamma)_i = \begin{cases}
1 &\text{if the} \ i^{\text{th}} \ \text{constraint in the Gamelin chain is a 2-point constraint}\\
\infty & \text{if the} \ i^{\text{th}} \ \text{constraint in the Gamelin chain is a Neil constraint}
\end{cases}
\]
Note that the parameter $\infty$ is interpreted as the constraint $f(a)=\infty \cdot f^{(n)}(a)$ -- equivalently, those functions such that $f^{(n)}(a)=0$. Therefore, the functions in the space $H^2_{D_{\Gamma}}$ are then functions that simply obey the constraints imposed on $A$. We quickly note here that, given any $f\in H^2_D$, if $f^{-1}$ exists, then $ff^{-1}=1\in H^2_{D_{\Gamma}}$.

The following lemma is now straightforward:

\begin{lemma} \label{landsinDgamma}
	If $f,g\in H^2_D$, then, provided $g^{-1}$ exists, $fg^{-1} \in H^2_{D_\Gamma}$
\end{lemma}


Finally we note that $\Sigma$ and $\Delta$, with their natural topologies, are compact metric spaces.

\section{The Szeg\H{o} theorem for constrained algebras} \label{lemmaforszego}

In this section we detail a few lemmas before exhibiting a proof of Theorem \ref{SzegoforA}. The first of which is straightforward to see:
\begin{lemma} \label{decomplemma}
	Given a real-valued $h\in H^2 \oplus \overline{H^2_0}$, if $h(x_0)=0$, then there exists $\xi \in H^2$ such that $h = \xi \oplus \xi^* \in H^2 + \overline{H^2_0}$. In particular, $\xi(x_0)=0$ as well.
\end{lemma}

\begin{lemma} \label{gamma=00} Suppose $\rho>0$ is a continuous function on $\partial X$. If $\textstyle \int_{\partial X} \log(\rho) \mathop{dm} = 0$, then there exists $\xi \in H^2$ and $n\in N$ such that $\log(\rho) = \xi \oplus \xi^* \oplus n \in H^2 \oplus \overline{H^2_0} \oplus N$, where $\xi(x_0) = 0$.
\end{lemma}
\begin{proof} To begin with, let $\log(\rho) = f \oplus g \oplus n \in H^2 \oplus \overline{ H^2_0  } \oplus N$. Let $P$ denote the orthogonal projection from $L^2$ onto $H^2$. Observe that, since $m$ is a representing measure for $x_0$ and $1 \in H^2$,
	\[
	\int_{\partial X} \log(\rho)\mathop{dm} = \langle \log(\rho),1 \rangle_2 = \langle \log(\rho), P1 \rangle_2 = \langle P\log(\rho),1\rangle_2 = \langle f, 1 \rangle_2 = \int_{\partial X} f \mathop{dm} =  f(x_0).
	\]
	Having assumed that $\textstyle \int_X \log(\rho)\mathop{dm} = 0$, it follows that $f(x_0)=0$. 
	
	Now, since $\log(\rho)$ is real-valued, we have that $f \oplus g \in H^2 \oplus \overline{H^2_0}$ is also real valued. By Lemma \ref{decomplemma}, there exists $\xi \in H^2_0$ such that $f\oplus g = \xi \oplus \xi^* \in H^2_0 \oplus \overline{H^2_0} \subseteq H^2 \oplus \overline{H^2_0}$.  Therefore
	\[
	\log(\rho) = \xi \oplus \xi^* \oplus n \in H^2 \oplus \overline{H^2_0} \oplus N
	\]
	with $\xi(x_0)=0$.
\end{proof}

\begin{lemma} \label{rhoreduction}
	Suppose $\rho>0$ is a continuous function on $\partial X$. Put $\widetilde{\rho} := e^c \rho$ for some constant $c$. Let $\xi, \zeta, n$ and $\widetilde{\xi}, \widetilde{\zeta}, \widetilde{n}$ be taken such that
	\[
	\log(\rho) = \xi \oplus \zeta \oplus n  \hspace{.25in} \text{and} \hspace{.25in} \log(\widetilde{\rho}) = \widetilde{\xi} \oplus \widetilde{\zeta} \oplus \widetilde{n},
	\]
	where both decompositions are occurring in $H^2 \oplus H^2_0 \oplus N$. If $D, \widetilde{D}\in \Delta$ are chosen so that $e^\xi$ and $e^{ \widetilde{\xi} }$ are in $H^2_{n,D}$ and $H^2_{\widetilde{n},\widetilde{D} }$ respectively, then $n=\widetilde{n}$ and $D = \widetilde{D}$ so that $H^2_{n,D} = H^2_{\widetilde{n},\widetilde{D} }$ and, in particular, $k^{n,D}_{x_0} = k_{x_0}^{ \widetilde{n},\widetilde{D}} $
\end{lemma}
\begin{proof} Since $\widetilde{\rho} = e^c \rho$, we have $\log(\widetilde{\rho}) = C + \log(\rho)$. Due to the assumed decompositions, we have
	\[
	\widetilde{\xi} \oplus \widetilde{\zeta} \oplus \widetilde{n} = \log( \widetilde{\rho} ) = C + \log(\rho) = (C + \xi) \oplus \zeta \oplus n.
	\]
	Since orthogonal decompositions are unique, we have $\widetilde{n}=n$ and $\widetilde{\xi} = C+\xi$. Let $D = (t_1,\ldots,t_d), \widetilde{D}= (\widetilde{t_1},\ldots, \widetilde{t_d}) \in \Delta$  as in the statement of the lemma. To argue that $D = \widetilde{D}$, it suffices to show that $t_i = \widetilde{t_i}$ for all $i$. To this end, recall that each of the $t_i$ are associated to either a 2-point or Neil constraint.
	
	Suppose first that $t_i$ exists such that, at the $i^{\text{th}}$ stage of the construction of $H^2_D$, we have
	\[
	H^2_{t_i}= \{k_a^{i-1} - t_i k_b^{i-1} \}^\perp = \{ f\in H^2_{t_i -1} \ : \ f(a)=t_i f(b) \},
	\]
	where $k_a^{i-1}$ and $k_b^{i-1}$ are the reproducing kernels in $H^2_{t_{i-1}}$ at $a$ and $b$ respectively. Then, since $D$ was chosen so that $e^\xi \in H^2_D$, we must have that $t_i = \exp( \xi(a)-\xi(b) )$. Likewise, $\widetilde{t_i} = \exp( \widetilde{\xi}(a) - \widetilde{\xi(b)} )$. However, since $\widetilde{\xi} = C+\xi$, it follows that
	\[
	\widetilde{\xi}(a) - \widetilde{\xi}(b) = \xi(a) +C - (\xi(b)+C  ) = \xi(a) - \xi(b)
	\]
	and therefore $\widetilde{t_i} = t_i$.

	Suppose instead that $t_i$ exists such that, at the $i^{\text{th}}$ stage of the construction of $H^2_D$, we have
	\[
	H^2_{t_i} = \{ k_a^{i-1} - t_i k^{i-1}_{a^{(n)}} \}^\perp = \{ f\in H^2_{t_i -1} \ : \ f(a)=t_if^{(n)}(a) \},
	\]
	where $k_a^{i-1}$ is the reproducing kernel in $H^2_{t_{i-1}}$ at $a$, and $k^{i-1}_{a^(n)}$ is the reproducing function in $H^2_{t_{i-1}}$ that returns a functions $n^{\text{th}}$ derivative at $a$. Since $D$ was chosen so that $e^\xi \in H^2_D$, we must have $\exp(\xi(a)) = t_i (\textstyle \frac{ d^n}{dx^n} \exp(\xi))\Big|_a$. Via repeated application of the chain and Liebniz rules, we find
	\[
	\frac{ d^n}{dx^n} \exp(\xi) = e^\xi \cdot G,
	\]
	where $G$ is a linear combination of products of $\xi',\ldots, \xi^{(n)}$. In particular, we find that 
	\[
	t_i = \frac{  \exp(\xi) }{ \textstyle \frac{ d^n}{dx^n} \exp(\xi) } \bigg\rvert_a =   \frac{  \exp(\xi) }{ \exp(\xi) \cdot G } \bigg\rvert_a = \frac{1}{G(a)}.
	\]
	Similarly, $\widetilde{t_i} = \textstyle \frac{1}{ \widetilde{G}    (a)}$ where $\widetilde{G}$ is a linear combination of products of $\widetilde{\xi}', \ldots \widetilde{\xi}^{(n) }$. Since $\widetilde{\xi} = C + \xi$, it follows that $\xi^{(j)} = \widetilde{\xi}^{(j)}$ for all $1\leq j \leq n$. This immediately implies that $G(a) = \widetilde{G}(a)$ so that $t_i = \widetilde{t_i}.$
	
	Having handled both cases, we conclude that $D = \widetilde{D}$. This fact, along with $\widetilde{n}=n$, allows us to conclude that $H^2_{n,D} = H^2_{\widetilde{n},\widetilde{D}}$ and therefore $k^{n,D}_{x_0} = k_{x_0}^{ \widetilde{n},\widetilde{D}} $.
\end{proof}

Before stating and proving the Szeg\H{o} theorem for $A$, we make a small observation. Given any function $\xi \in H^2$, we know that $\xi$ is bounded below on $X$, and therefore $e^\xi$ will never be zero for any points in $X$. Due to the nature of the Neil and 2-point constraints, the only functions that can live in two different $H^2_D$ spaces are those whose constrained values vanish. In other words, functions $f$ for which $f_\Gamma = (0,\ldots,0)^\top \in \C^\gamma$. Due to this fact, the function $e^\xi$ cannot live in two different $H^2_D$ spaces. This justifies the notion that there exists a unique tuple $D\in \Delta$ for which $e^\xi \in H^2_D$.

\noindent\textbf{Theorem \ref{SzegoforA}} (Szeg\H{o} Theorem for $A$)
\textit{Suppose $\rho > 0$ is a continuous function on $\partial X$. Let $\xi \in H^2, \zeta \in \overline{H^2_0}, n\in N$ be functions such that $\log(\rho) = \xi \oplus \zeta \oplus n \in H^2 \oplus \overline{H^2_0} \oplus N$. Let $D \in \Delta$ be the unique tuple such that $e^\xi \in H^2_{n,D}$. With $C_\rho := \textstyle \int_{\partial X} \log(\rho) \mathop{dm}$, it follows that
	\[
	\inf\left\{  \int_{\partial X} |1-p|^2 \rho \mathop{dm} \ : \ p\in  A_0  \right\} = \exp(C_\rho) \left(  \frac{1}{ \|k^{n,D}_{x_0}\|^2  }  \right).
	\]}
\begin{proof} We begin by observing that it suffices to consider $C_\rho = 0$. If not, we consider $\widetilde{\rho} = \exp(-C_\rho)\rho$. We have that $\log(\widetilde{\rho}) = -C_\rho + \log(\rho)$. We see immediately that
	\[
	C_{\widetilde{\rho}} = \int_{\partial X} \log(\widetilde{\rho}) \mathop{dm} = \int_{\partial X} -C_\rho + \log(\rho) \mathop{dm} = -C_\rho +C_\rho = 0.
	\]
	Thus, provided we establish the result for $C_{\widetilde{\rho}} = 0$, we have that
	\begin{equation}\label{tilde}
	\inf\left\{  \int_{\partial X} |1-p|^2 \widetilde{\rho} \mathop{dm} \ : \ p\in  A_0  \right\} = \exp(C_{\widetilde{\rho}}) \left(  \frac{1}{ \|k_{x_0}^{ \widetilde{n},\widetilde{D}    }\|^2  }  \right)
	\end{equation}
	where $\widetilde{n} \in N$ and $\widetilde{D}\in \Delta$ are the unique vectors such that $\log(\widetilde{\rho}) = \widetilde{\xi} \oplus \widetilde{\zeta} \oplus \widetilde{n} \in H^2 + \overline{H^2_0} \oplus N$ and $e^{\widetilde{\xi}}\in H^2_{\widetilde{n},\widetilde{D} }$. Since $-C_\rho$ is a constant, it follows from Lemma \ref{rhoreduction} that $k^{n,D}_{x_0} = k_{x_0}^{ \widetilde{n},\widetilde{D}} $. Therefore, since $C_{\widetilde{\rho}} = 0$, (\ref{tilde}) becomes
	\begin{equation*}
	\inf\left\{  \int_{\partial X} |1-p|^2 \widetilde{\rho} \mathop{dm} \ : \ p\in  A_0  \right\} = \frac{1}{ \|k^{n,D}_{x_0}\|^2  }.
	\end{equation*}
	Thus,
	\begin{align*}
	\inf\left\{  \int_{\partial X} |1-p|^2 \rho \mathop{dm} \ : \ p\in A_0 \right\} &= \exp(C_\rho)\inf\left\{  \int_{\partial X} |1-p|^2 \widetilde{\rho} \mathop{dm} \ : \ p\in  A_0  \right\}\\
	&= \exp(C_\rho) \left(\frac{1}{ \|k^{n,D}_{x_0}\|^2  }\right).
	\end{align*}
	
	Henceforth, we assume that $C_\rho = 0$. In view of Lemma \ref{gamma=00},  there exist unique $\xi \in H^2$ and $n\in N$ such that $\log(\rho) = \xi \oplus \xi^* \oplus n \in H^2 \oplus \overline{H^2_0} \oplus N$ with $\xi(x_0)=0$. Define the space $H^2(\rho)$ to be the standard $H^2$ space but with the inner product given by
	\[
	\langle f,g \rangle_\rho = \int_X f\overline{g} \mathop{\rho dm}.
	\]
	Let $\overline{A_0}^{\|\cdot\|_{H^2(\rho)}}\subseteq H^2(\rho)$ denote the $L^2(\rho)$ closure of $A_0$. It suffices to argue that the $H^2(\rho)$- distance from the vector 1 to the space $\overline{A_0}^{\|\cdot\|_{H^2(\rho)}}$ is equal to $\textstyle \frac{1}{ \|k^{n,D}_{x_0}\|^2  }$. 
	
	By exponentiating, we have that $\rho = e^\xi e^{\xi^*} e^n$. We claim that $e^\xi \in H^\infty$. Indeed,
	\[
	|e^\xi|^2 = e^\xi e^{\xi^*} = \exp(\xi + \xi^*) = \exp(\log(\rho)-n) = \rho e^{-n}.
	\]
	Since $\rho$ and $n\in N$ are both bounded on $\partial X$, the above shows that $|e^\xi|^2$ is bounded and therefore $e^\xi \in L^\infty$. However, since $\xi \in H^2$, it follows that $e^\xi \in H^2$ as well and therefore $e^\xi \in H^\infty$. Similarly, $e^{-\xi} \in H^\infty$.
	
	Define the map $U\colon H^2(\rho)\to H^2_n$ by $f\mapsto e^\xi f$. (Since $e^\xi$ is bounded, this map is both well defined and bounded.) Observe that $n\in N$ is the unique value that makes $U$ act as an isometry from $H^2(\rho)$ to $H^2_n$. That is, given $f \in \overline{A_0}^{\|\cdot\|_{H^2(\rho)}}$, we have that
	\[
	\|Uf\|_n^2 = \|e^\xi f\|_n^2 = \int_{\partial X} |e^\xi f|^2 \mathop{e^n dm} = \int_{\partial X} |f|^2 \mathop{e^\xi e^{\xi^*}e^n dm} = \int_{\partial X} |f|^2 \mathop{\rho dm} = \|f\|_{H^2(\rho)}^2.
	\]
	In particular, defining its inverse by $U^*\colon H^2_n \to H^2(\rho)\colon f \mapsto e^{-\xi}f$, we find that this defines an isometry as well. Thus $U$ is a unitary between $H^2(\rho)$ and $H^2_n$. Additionally, recall that $D \in \Delta$ was chosen specifically so that $e^\xi \in H^2_{n,D}$. Now, since $U$ is surjective, we have that $U( \overline{A_0}^{\|\cdot\|_{H^2(\rho)}} ) = \{ f\in H^2_{n,D} \ : \ f(x_0)=0 \} =: H^{2}_{n,D; 0}$ is exactly those functions in $H^2_{n,D}$ that vanish at $x_0$. Since $U(1) = e^\xi$, we can transport our question over to the $H^2_{n,D}$ setting and observe that it suffices to show that the $H^2_{n,D}$-distance from $e^\xi$ to $H^{2}_{n,D;0}$ is exactly $\textstyle \frac{1}{ \|k^{n,D}_{x_0}\|^2  }$.
	
	Recall that the assumption $C_\rho =0$ yields $\xi(x_0) = 0$. Therefore $e^{\xi(x_0)} = 1 \neq 0$ and hence $e^\xi \notin H^{2}_{n,D;0}$.
	Further, we know that $H^{2}_{n,D;0}$ is a codimension 1 subspace of $H^2_{n,D}$ and, in particular, $H^{2}_{n,D;0} = (\spann\{k^{n,D}_{x_0}\})^\perp$, where $k^{n,D}_{x_0}$ is the reproducing kernel for $H^2_{n,D}$ at $x_0$.
	
	Since $H^{2}_{n,D;0}$ is a closed subspace, there exists $f\in H^{2}_{n,D;0}$ that minimizes $\|e^\xi -f\|^2_n$. This $f$ is exactly $f = \text{proj}_{H^{2}_{n,D;0}}(e^\xi)$. Observe that $(e^\xi - f) \perp H^{2}_{n,D;0}$ and therefore $(e^\xi - f) \in \spann\{k^{n,D}_{x_0}\}$. Thus,
\[
	e^\xi -f = \text{proj}_{(H^{2}_{n,D;0})^\perp}(e^\xi - f) = \text{proj}_{(H^{2}_{n,D;0})^\perp}(e^\xi) = e^{\xi(x_0)}\frac{k^{n,D}_{x_0}}{\|k^{n,D}_{x_0}\|^2}= \frac{k^{n,D}_{x_0}}{\|k^{n,D}_{x_0}\|^2}.
\]
	Therefore $\|e^\xi - f\|^2_n = \textstyle \frac{1}{\|k^{n,D}_{x_0}\|^2}$. But this is exactly the $H^2_{n,D}$-distance from $e^\xi$ to the space $H^{2}_{n,D;0}$. We saw earlier that this distance is equal to the desired $H^2(\rho)$-distance from the vector 1 to the space $\overline{A_0}^{\|\cdot\|_{H^2(\rho)}}$. Hence the proof is complete.
\end{proof}

\begin{remark} \label{Douglasremark}
    We note that the computation carried out in this proof, when restricted to the classical setting (with suitable notation adjusted appropriately), does the heavy lifting in establishing a characterization for outer functions in $H^2(\T)$. Namely, if $A_0$ are those functions in the disc algebra that vanish at zero, and $\hat{f}$ denotes the holomorphic extension of $f$ over $\D$ obtained via the Poisson kernel, then $f\in H^2(\T)$ is outer if and only if 
    \[
    \inf_{h\in A_0}\left\{ \frac{1}{2\pi} \int_0^{2\pi} |1-h|^2|f|^2\mathop{d\theta}\right\} = \left| \hat{f}(0)\right|^2.
    \]
    This characterization can be found as Exercise 6.28 in \cite{BanachTech}.
\end{remark}


\section{Invertibility of Toeplitz operators: some preliminary lemmas} \label{lemmassection}

We begin by noting that most of the arguments given here will work in the full setting of a finite (connected) Riemann surface, but one needs to modify the instances involving the Green's function (e.g., arguments given in Subsection \ref{FactorizationSection}). Thus, for simplicity, we instead consider the underlying domain to be a $\tau$-holed planar domain. In this subsection, details on the Green's function for planar domains are reproduced largely from \cite{AbrahamseToeplitz} and \cite{Fisher}.

\subsection{Green's Function for Planar Domains} \label{greensection}

Let $X$ be a $\tau$-holed planar domain with $x_0\in X$. The \textit{Green's function of $X$ with pole at $x_0$} is defined by
\[
G(z,x_0) = -\log(|z-x_0|) + h(z,x_0),
\]
where $h(z,x_0)$ is the unique harmonic function of $z$ in $X$ with boundary values given by $\log(|z-x_0|)$. Such an $h$ exists because the Dirichlet problem is solvable on $X$ (the fact that $h$ is unique is guaranteed by the maximum principle for harmonic functions). 

Equivalently, the Green's function is the unique function that satisfies the following properties:
\begin{enumerate}[(i)]
    \item $G(z,x_0)$ is harmonic on $X\setminus \{x_0\}$
    \item $G(z,x_0) + \log(|z-x_0|)$ is harmonic near $x_0$
    \item $G(z,x_0) \to 0$ as $z\to \partial X$
\end{enumerate}

The Green's function allows us to pass between the representing measure $dm$ and the arclength measure $dz$. Let $H$ be the multi-valued harmonic conjugate of $-G$ and let $v$ denote the single-valued derivative of $-G+iH$. Then: 

\begin{proposition}[Reformulated Proposition 6.5 in \cite{Fisher}]\label{greenmeasure} In the notation just introduced,
\[
\mathop{dm(z)} = \frac{1}{2\pi i} v(z) \mathop{dz}
\]
\end{proposition}
Note that the above proposition can also be found in the discussion immediately proceeding Proposition 1.3 in \cite{AbrahamseToeplitz}. One of the best uses of the above proposition is using $v$ to characterize the $N$-space discussed in Subsection \ref{hDAlg}. Specifically, if $S_{x_0}$ denotes the real linear span of the set of all differences between pairs of representing measures for $x_0$, then each $\mu_i\in S_{x_0}$ is absolutely continuous with respect to $dm$. Putting $\lambda_i := d\mu_i/dm$, we define
\[
N := \spann_{\C}\{ \lambda_1,\ldots,\lambda_\sigma\}.
\]
This space turns out to `fill out' $L^2$. Specifically, as noted in Section 4.5 of \cite{Fisher} and Section 2 of \cite{AbrahamseToeplitz}, we have
\[
L^2(dm) = H^2(\partial X) \oplus \overline{ H^2_0(\partial X)} \oplus N.
\]
Moreover, the following theorem relates this decomposition to the derivative of the Green's function, $v$:
\begin{proposition}[Theorem 1.7 in \cite{AbrahamseToeplitz}]\label{greenortho} The orthogonal complement of $H^2(\partial X)$ in $L^2(dm)$ is $\overline{v}^{-1} \overline{H^2(\partial X)}$. Therefore $L^2(dm) = H^2(\partial X) \oplus \overline{v}^{-1} \overline{H^2(\partial X)}.$
\end{proposition}
The last result we need involving the Green's function is the information it encodes about the domain $X$.
\begin{proposition}[Reformulated Proposition 1.4 in \cite{AbrahamseToeplitz}]\label{greenzeros} If $X$ is a $\tau$-holed planar domain, then the following is true of the function $v$:
\begin{enumerate}[(i)]
    \item It is meromorphic in a neighborhood of $\overline{X}$ with exactly one pole of order one at $x_0$ and no other poles.
    \item It has precisely $\tau$ zeros in $X$, counting multiplicities, and no other zeros in $\overline{X}$
\end{enumerate}
\end{proposition}
A version of the above proposition can also be found as Proposition 6.5 in \cite{Fisher}.

\subsection{Inner functions, the norm of Toeplitz operators, and kernels in $H^2_{\alpha,D}$}

As mentioned in Subsection \ref{hDAlg}, we will refer to a function $h \in H^p$ as an \textit{inner function} if there exists an $\alpha \in \Sigma$ such that $|h| = |Z|^\alpha$. In this manner, inner functions act as isometric multipliers between $L^p_\alpha$ spaces for different $\alpha$'s. A function $g \in H^p$ is an \textit{outer function} if  $\textstyle \log(| \int_{\partial X} g \mathop{dm}|  ) =  \int_{\partial X } \log(|g|)\mathop{dm} > -\infty$.

\begin{lemma} \label{innerexist}
	There exists an inner function $\Phi \in  H^2$ such that $\Phi_\Gamma = (0,\ldots,0)^\top \in \C^\gamma$. In particular, $\Phi \in H^2_{\alpha,D}$ for every $(\alpha,D) \in \Sigma \times \Delta$.
\end{lemma}
\begin{proof}
	The $i^{\text{th}}$ entry of the vector $\Phi_\Gamma \in \C^\gamma$ is of the form $\Phi^{(n_i)}(a_i)$ for some $a_i\in X$ and $n_i\geq 0$. Since $A$ is a uniform algebra, we can find an $f_i \in A$ such that $f_i^{(n_i)}(a_i)=0$. Note that, since we have an algebra over a $\tau$-holed planar domain, we can choose $f_i$ in such a manner that $\log(|f_i|)$ is integrable with respect to all measures in $M_{x_0}$.  Now, since $H^2$ is defined to be the $L^2$ closure of $A$, we have that $f_i \in H^2$. By Theorem \ref{ahernfactorization}, there exists $H^2$ functions $g_i$ and $h_i$ such that $g_i$ is outer, $h_i$ is inner, $f_i = g_ih_i$, $\textstyle \int_{\partial X} \log(|g_i|)\mathop{d\nu} = (0,\ldots,0)$, and  $h_i^{(n)}(a_i)=0$. Since $h_i$ is inner, there exists a $\gamma$-tuple $\alpha_i$ such that $|h_i| = |Z|^{\alpha_i}$. 
	
	Doing the above for every $a_i$ we then form $\Phi = \textstyle \prod_{i=1}^\gamma h_i$ and $\alpha' = \sum_{i=1}^\gamma \alpha_i.$ In this manner, $\Phi$ is an inner function in $H^2$ such that $\Phi_\Gamma = (0,\ldots,0)^\top$ and $|\Phi| = |Z|^{\alpha'}$. Technically, every $H^2_\alpha$ is the same set of functions for every $\alpha$. Thus, $\Phi \in H^2_\alpha$ for every $\alpha$. Note further that, since $\Phi_\Gamma = (0,\ldots,0)^\top$, we have that $\Phi \in H^2_{\alpha,D}$ for every $(\alpha,D) \in \Sigma \times \Delta$.
\end{proof}

\begin{lemma} \label{adjointoftoeplitz}
	If $\phi \in L^\infty$, then $\| T^{\alpha,D}_\phi \| = \|\phi \|$ and $( T^{\alpha,D}_\phi )^* = T^{\alpha,D}_{\overline{\phi}}$.
\end{lemma}
\begin{proof} Since $M^*_\phi = M_{\overline{\phi}}$, we have
	\[
	(T^{\alpha,D}_\phi)^* = V^*_{\alpha,D} M^*_\phi V_{\alpha,D} = V^*_{\alpha,D} M_{\overline{\phi}} V_{\alpha,D} = T^{\alpha,D}_\phi.
	\]
	Since $V_{\alpha,D}$ is an isometry,
	\[
	\| T^{\alpha,D}_\phi \| \leq \|V^*_{\alpha,D} \| \|M_\phi \| \|V_{\alpha,D}\| \leq \|M_\phi \| = \|\phi\|.
	\]
	Thus, it suffices to show that $\|T^{\alpha,D}_\phi \| \geq \|\phi\|$. To this end, let $\Phi$ be the inner function from Lemma \ref{innerexist} such that $\Phi_\Gamma = (0,\ldots,0)^\top$. If we denote by $H^2$ the unweighted $H^2_\beta$ space and let $\beta \in \Sigma$ be the $\sigma$-tuple such that $|\Phi | = |Z|^{-\beta}$, then it follows that $\Phi$ is an isometric multiplier from $H^2$ into $H^2_{\beta, D}$ for every $D \in \Delta$.
	
	Now, denoting by $L^2$ the unweighted $L^2_\alpha$ space, let $V\colon H^2 \to L^2$ denote the inclusion map. Likewise, let $W\colon \Phi H^2 \to L^2_\beta$ be the inclusion map. (Note here that with this setup, $V^* M_\phi V = T_\phi$ is the usual Toeplitz operator on $H^2$.) Now, let $\Psi \in H^\infty \subseteq H^2$ be the inner function given by $|\Psi| = |Z|^{\beta - \alpha}$ (such a function exists by Lemma \ref{ahernlemma}). Observe that $\Psi \Phi$ is an isometric multiplier from $L^2$ to $L^2_\alpha$. Further,
	\begin{align} \label{unitequiv}
	W^* M_\phi W = W^*\Psi^* V_{\alpha,D}^* M_\phi V_{\alpha,D} \Psi W = W^* \Psi^* T^{\alpha,D}_\phi \Psi W
	\end{align}
	
	Now define the map $U\colon H^2 \to \Phi H^2 \subseteq H^2_{\beta,D}$ sending $f \mapsto \Phi f$. As noted earlier, this map is an isometry into $H^2_{\beta,D}$. Thus, for $f,g \in H^2$,
	\begin{align*}
	\langle M_\phi W Uf, W U g \rangle_{H^2_\beta} &= \langle M_\phi W \Phi f, W\Phi g \rangle_{H^2_\beta}\\
	&= \langle M_\phi \Phi f, \Phi g \rangle_{L^2_\beta}\\
	&= \langle \Phi \phi f, \Phi g \rangle_{L^2_\beta}\\
	&= \langle \phi f, g \rangle_{L^2}\\
	&= \langle M_\phi f,g \rangle_{L^2}\\
	&= \langle M_\phi f, Pg \rangle_{L^2}\\
	&= \langle T_\phi f, g \rangle_{H^2}.
	\end{align*}
	Therefore $U^*(W^* M_\phi W)U = T_\phi$. Combining this with (\ref{unitequiv}), we have that
	\[
	U^*(   W^* \Psi^* T^{\alpha,D}_\phi \Psi W    )U = T_\phi.
	\]
	It follows that,
	\[
	\|T_\phi\| = \| U^*(   W^* \Psi^* T^{\alpha,D}_\phi \Psi W    )U \| \leq \|T^{\alpha,D}_\phi \|
	\]
	and therefore $\|\phi\| = \|T_\phi \| \leq \|T^{\alpha,D}_\phi \|$.
	
\end{proof}

Let $w_i^\alpha \in H^2_{\alpha}$ be the linear combination of reproducing functions in $H^2_{\alpha}$ such that for all $f \in H^2_{\alpha}$, either $\langle f, w_i^\alpha \rangle_\alpha = f(a)-t_if(b)$ or $\langle f, w_i^\alpha\rangle_\alpha = f(a)-t_i f^{(n)}(a)$. It follows that
\[
\spann\{  w_1^\alpha,\ldots,w_d^\alpha \} = (H^2_{\alpha,D})^\perp.
\]

Further, let $h^\alpha_i \in H^2_{\alpha}$ be the reproducing functions such that $\langle  f, h^\alpha_i \rangle_\alpha$ returns either $f(a)$ or $f^{(n)}(a)$ (for some finite $n$), depending on how $a$ is integrated into the construction of $A$. In this manner,
\[
f_\Gamma = \begin{bmatrix}
\langle f, h_1^\alpha \rangle_\alpha \\
\vdots\\
\langle f, h^\alpha_\gamma \rangle_\alpha
\end{bmatrix} \in \C^\gamma.
\]An immediate observation is that $w^\alpha_i \in \spann\{h_1^\alpha,\ldots,h_\gamma^\alpha \}$ for every $1\leq i \leq d$. This discussion is recorded in the following proposition:

\begin{proposition} \label{perpspacespann}
	Given an $(\alpha,D)\in \Sigma \times \Delta$, there exists reproducing functions $h_1^\alpha, \ldots, h_\gamma^\alpha$ and $w_1^\alpha,\ldots,w_d^\alpha \in H^2_{\alpha}$ such that
	\begin{enumerate}
		\item $
		f_\Gamma = \begin{bmatrix}
		\langle f, h_1^\alpha \rangle_\alpha \\
		\vdots\\
		\langle f, h^\alpha_\gamma \rangle_\alpha
		\end{bmatrix} \in \C^\gamma$ and
		\item $(H^2_{\alpha,D})^\perp = \spann\{ w_1^\alpha,\ldots, w_d^\alpha\}$,
	\end{enumerate}
	where $w^\alpha_i \in \spann\{h_1^\alpha,\ldots,h_\gamma^\alpha \}$ for every $1\leq i \leq d$.
\end{proposition}



\subsection{Pre-Annihilator for $A$ and a Factorization Lemma} \label{FactorizationSection}

 In this subsection we begin by asserting the existence of a space $\mathscr{M}$ such that $L^\infty / A$ is isometrically isomorphic to $\mathscr{M}^*$. Due to the Hahn-Banach Theorem, this goal is equivalent to finding $\mathscr{M} \subseteq L^1$ such that $\text{Ann}(\mathscr{M}) = A$. We will continue to use the machinery introduced and discussed in Subsection \ref{greensection}. Specifically, recall that, since our uniform algebra $\mathscr{A}$ is a hypo-Dirichlet algebra, the real linear span of the set of differences between pairs of representing measures for $x_0$, denoted $S_{x_0}$, was $\sigma$-dimensional. Putting $S_{x_0} = \spann_{\R}\{  \mu_1,\ldots, \mu_\sigma \}$, we noted that each $\mu_i$ is absolutely continuous with respect to $\mathop{dm}$. Thus, with $\lambda_i := \mathop{d\mu_i}/\mathop{dm}$, we can form $N=\spann_\C\{\lambda_1,\ldots,\lambda_\sigma\}$. Specifically, this $N$ space exists such that $L^2 = H^2 \oplus \overline{ H^2_0} \oplus N.$ In \cite{AhernSarason}, it is shown that $N$ has a basis consisting of real functions so that, as sets, $\overline{N}=N$. Further, the aforementioned $L^2$ decomposition can be weighed:
\begin{equation}\label{L2decomp}
L^2_\alpha = H^2_\alpha \oplus \overline{H^2_{0,\alpha}} \oplus N.
\end{equation}
where $\overline{H^2_{0,\alpha}}$ is the complement of $H^2_\alpha$ with the additional condition that $\textstyle \int_X f \mathop{dm} = f(x_0) = 0$. Note further that when we consider the unweighted $L^2$ space ($\alpha = 0$), the above is written without the subscript adornment: $\overline{H^2_0}$.

 Now, observe that if $\phi \in A$, then we have that
\[
\int_{\partial X} \phi \lambda_i \mathop{dm} = \int_{\partial X} \phi \frac{ \mathop{d\mu_i}}{ \mathop{dm} }\mathop{dm} = \int_{\partial X} \phi \mathop{ d\mu_i } = 0.
\]
Thus, every function in $N$ is annihilated by $\phi\in A$.

Proposition \ref{perpspacespann} asserted the existence of reproducing functions $h_1,\ldots,h_\gamma \in H^2$ such that, given $f\in H^2$, 
\[
f_\Gamma = \begin{bmatrix}
\langle f, h_1 \rangle_2 \\
\vdots\\
\langle f, h_\gamma \rangle_2
\end{bmatrix} \in \C^\gamma.
\] Recall that the construction of the finite codimensional subalgebra $A$ started with the uniform algebra $\mathscr{A}$ and iteratively imposed either 2-point or Neil constraints. Recall that $d \leq \gamma$ denoted the number of iterations necessary to yield our algebra $A$. For $1\leq i \leq d$, let $s_i$ denote the linear combination of vectors in $\spann\{ h_1,\ldots, h_\gamma \}$ such that $\langle \phi, s_i \rangle_2 = 0$ for all $\phi \in A$. This directly implies that, given $\phi \in A$,
\[
\int_{\partial X} \phi \overline{s_i} \mathop{dm} = \langle \phi, s_i \rangle_2 = 0
\]
for every $1\leq i \leq d$. With $\mathcal{S} := \spann\{ \overline{s_1}, \ldots, \overline{s_d} \}$, we've shown that every function in $\mathcal{S}$ is annihilated by $\phi\in A$.

\begin{example}
	Suppose we start with a uniform algebra $\mathscr{A}$ and impose a single 2-point constraint at $a,b\in X$. If we let $h_a$ and $h_b$ be the $H^2$ functions that reproduce at the points $a$ and $b$ -- that is, $f(a) = \langle f, h_a \rangle_2$ and $f(b) = \langle f, h_b \rangle_2$, then $s_{a,b} := h_a - h_b$. In this manner, we find that all functions $\phi \in A$ must obey $\phi(a) = \phi(b)$ and thus,
	\[
	0 = \phi(a)-\phi(b) = \langle \phi, h_b \rangle_2 - \langle \phi, h_b \rangle_2 = \langle \phi, h_a - h_b \rangle_2 = \langle \phi, s_{a,b} \rangle_2.
	\]	
\end{example}

Finally, recall that $H^1$ is defined to be the $L^1$ closure of the algebra $\mathscr{A}$. If we put $H_0^1 := \{  f\in H^1 \ : \ \textstyle \int_X f\mathop{dm} = f(x_0) = 0  \}$, then we find that, given $\phi \in A$,
\[
\int_{\partial X} \phi f \mathop{dm} = (\phi f)(x_0) = 0
\]
for all $f \in H^1_0$. Therefore every function in $H_0^1$ is annihilated by $\phi\in A$.  Put $\mathscr{M}=H_0^1 + N + \mathcal{S}$.

\begin{lemma} \label{annihilatorlemma}
	$\text{Ann}(\mathscr{M}) = A$.
\end{lemma}
\begin{proof}
	As noted in the previous discussion, we know that $H_0^1 + N + \mathcal{S}= \mathscr{M}$ is annihilated by any function $\phi \in A$. This shows that $A \subseteq \text{Ann}(\mathscr{M})$. To argue the other inclusion, since $A \subseteq L^\infty$, it is enough to show that if $\phi \in L^\infty$ such that $\textstyle \int_X \phi h \mathop{dm} = 0$ for every $h\in \mathscr{M}$, then $\phi \in A$. To this end, we first note that our algebra $A$ can be interpreted as
	\begin{equation} \label{HinftyformulationforA}
	A = \{  f\in H^\infty \ : \ f \ \text{satisfies the constraints of} \ A \}.
	\end{equation}
	 Since $\phi \in L^\infty$, we also have that $\phi \in L^2$.  Recall from (\ref{L2decomp}),
	\[
	L^2 = H^2 \oplus \overline{H^2} \oplus N.
	\]
	Additionally, recall that $N=\overline{N}$. Since $\langle \phi, \lambda \rangle_2 =  \textstyle \int_{\partial X} \phi \lambda \mathop{dm}=0$ for every $\lambda \in N$, we must have $\phi \in H^2 \oplus \overline{H^2_0}$. Further, since $H^2_0 \subseteq H^1_0$ and  $\textstyle \int_{\partial X} \phi h \mathop{dm} = 0$ for every $h\in H^1_0$, it follows that $\langle \phi, \overline{h} \rangle_2 = \textstyle \int_{\partial X} \phi h \mathop{dm} = 0$ for every $h\in H^2_0$ as well. Hence $\phi \perp \overline{H^2_0}$. This puts $\phi \in H^2$. Since $H^\infty = H^2 \cap L^\infty$ and $\phi\in L^\infty$, it follows that $\phi \in H^\infty$.
	
	Lastly, since $\textstyle \int_{\partial X} \phi \overline{s_i} \mathop{dm}=0$ for every $1\leq i \leq d$, this implies (by the construction of the functions $s_i$) that $\phi$ must satisfy the constraints of $A$. Thus, using the formulation given in (\ref{HinftyformulationforA}), we find that $\phi \in A$.
\end{proof}

\begin{proposition} \label{quotientisomorphictoM}
	$L^\infty / A \cong \mathscr{M}^*$.
\end{proposition}
\begin{proof}
	By the Hahn-Banach Theorem, $X^* / \text{Ann}(Y) \cong Y^*$. In our case, we know that $(L^1)^* \cong L^\infty$. Thus, since Lemma \ref{annihilatorlemma} showed that $A = \text{Ann}(\mathscr{M})$, it follows that $L^\infty / A \cong \mathscr{M}^*$ as desired. In particular, if we let $\Lambda$ be the isometric isomorphism from $L^\infty /A$ to $\mathscr{M}^*$, then it is interpreted as the map sending $\pi(\phi)\mapsto (\lambda_\phi)\big|_{\mathscr{M} }$, where $\lambda_\phi\colon L^1\to \C$ is the functional sending $\psi \mapsto \textstyle \int_{\partial X} \phi \psi\mathop{dm}$.
\end{proof}

Recalling the notation from Subsection \ref{greensection}, let $G$ be the Green's function for the $\tau$-holed planar domain $X$ with pole at $x_0$, $H$ be the multi-valued harmonic conjugate of $-G$, and $v$ be the single-valued derivative of $-G+iH$.

The following lemma is a weighted version of Proposition \ref{greenortho}:

\begin{lemma} \label{lemma1}
	The orthogonal complement of $H_\alpha^2$ in $L_\alpha^2$ is $\overline{v}^{-1}\overline{H_{0,\alpha}^2}$.
\end{lemma}

In light of the decomposition $L_\alpha^2 = H_\alpha^2 \oplus \overline{H^2_{0,\alpha}} \oplus N$, and recalling that (as sets) $N = \overline{N}$, the above lemma implies that
\[
H^2_{0,\alpha} \oplus N = v^{-1}H^2_\alpha.
\]
Hence $\mathscr{M} = H_0^1 + N + \mathcal{S} = v^{-1}H^1 + \mathcal{S}$.

Let $\mathcal{M}  \subseteq H^1$ be the dense subset of functions analytic on $X$, continuous on $\partial X$ and with no zeros on $\partial X$. Note that $\mathcal{M}$ is dense in $H^2$ as well. With this, let 
\[
\mathscr{M}_\mathcal{M} := v^{-1}\mathcal{M} + \mathcal{S}.
\]
This set is dense in $\mathscr{M}$. We prove a factorization theorem for $\mathscr{M}_\mathcal{M}$. Before we do so, however, we need a technical lemma. For a $\omega \in \Sigma$, let $H^1_\omega$ denote the usual $H^1$ space but with norm $\|f\|_{1,\omega} = \textstyle \int_{\partial X} |f| |Z|^\omega \mathop{dm}$. (In this manner, we will let the subscript denote the fact that this is the 1-norm. Similarly, $\|\cdot\|_{2,\omega}$ will denote the 2-norm coming from the inner product discussed in (\ref{H2alphainnerproduct}) but with $\omega \in \Sigma$. )

\begin{lemma} \label{babyfactorization}
	Given $h\in \mathcal{M} \subseteq H^1_\gamma$ for some $\gamma \in \Sigma$, there exists $\omega \in \Sigma$ and $F,G \in H^2$ such that $h = FG$, $\|h\|_{1,\gamma} = \|F\|_{2,\gamma - \omega} \|G\|_{2, \gamma + \omega}   $, $F$ is invertible, and $F \in \mathcal{M}$.
\end{lemma}
\begin{proof}
	Let $h\in \mathcal{M} \subseteq H^1_\gamma$. It follows from the definition of $\mathcal{M}$ that $h$ has at most finitely many zeros in $X$. The proof of Lemma \ref{innerexist} guarantees the existence of an inner function $g$ that shares exactly these finitely manner zeros (with multiplicity). Let $\omega \in \Sigma$ be taken such that $|g| = |Z|^\omega$. Putting $f:= h/g$, we have that $f$ is analytic in $X$, has no zeros on $\overline{X}$ and is continuous on $\partial X$ -- therefore $f$ is invertible. In particular, we find that $\sqrt{f}$ is also invertible and lies in $\mathcal{M}\subseteq H^2$.
	
	Observe that
	\begin{align*}
	\| h \|_{1,\gamma} &= \int_{\partial X} |h| |Z|^\gamma \mathop{dm}= \int_{\partial X} |g| |f| |Z|^\gamma \mathop{dm}=\int_{\partial X} |f| |Z|^{\gamma + \omega} \mathop{dm}= \|f\|_{1,\gamma + \omega}= \| \sqrt{f} \|^2_{2,\gamma +\omega}.
	\end{align*}
	Further,
	\[
	\|  g\sqrt{f} \|^2_{2,\gamma-\omega} = \int_{\partial X} ( |g| |\sqrt{f}| )^2 |Z|^{\gamma-\omega} \mathop{dm} = \int_{\partial X} |\sqrt{f}|^2 |Z|^{\gamma+\omega} \mathop{dm} =  \|\sqrt{f}\|^2_{2,\gamma+\omega}.
	\]
	Therefore
	\[
	\|h\|_{1,\gamma} = (\|h\|_{1,\gamma})^{\frac{1}{2}}(\|h\|_{1,\gamma})^{\frac{1}{2}} = \| \sqrt{f} \|_{2,\gamma+\omega}  \|  g\sqrt{f} \|_{2, \gamma-\omega}. 
	\]
	Putting $F = \sqrt{f}$ and $G= g\sqrt{f}$, we find that $h = FG$, $\|h\|_{1,\gamma} = \|F\|_{2,\gamma - \omega} \|G\|_{2, \gamma + \omega}   $, $F$ is invertible, and $F \in \mathcal{M}$.
\end{proof}

\begin{proposition} \label{factorizationforMp}
	For $h\in \mathscr{M}_\mathcal{M} = v^{-1}\mathcal{M} + \mathcal{S}$, there exists $(\beta,D)\in \Sigma \times \Delta$, $f\in H^2_{-\beta,D}$, and $g\in L^2_{-\beta}$ such that
	\begin{enumerate}[(i)]
		\item \hspace{.25in} $h=fg$
		\item  \hspace{.25in} $\|h\|_1 = \|f\|_{2,-\beta}\|g\|_{2,\beta}$
		\item  \hspace{.25in} $\langle \psi, \overline{g} \rangle_2 = 0$ for all $\psi\in  H^2_{-\beta,D}$
	\end{enumerate}
\end{proposition}
\begin{proof} 
	Let $h = v^{-1}r + h_s \in \mathscr{M}_{\mathcal{M}} = v^{-1}\mathcal{M} + \mathcal{S}$. Let $z_1,\ldots,z_\tau$ be the zeros of $v$ in $X$. The proof of Lemma \ref{innerexist} guarantees the existence of an inner function $\Phi_v\in H^2$ such that $\Phi_v(z_i)=0$ for $1\leq i \leq \tau$. Let $\Phi^{\Gamma}$ be the actual function produced by Lemma \ref{innerexist} so that $\Phi_\Gamma^\Gamma = (0,\ldots,0)^\top \in \C^\gamma$. It follows that the product $\Phi:= \Phi_v\Phi^\Gamma$ is inner as well and thus there exists $\alpha\in \Sigma$ such that $|\Phi| = |Z|^\alpha$.
	
	Our first order of business will be showing that
	\begin{enumerate}[(a)]
		\item  \hspace{.25in} $\Phi h \in H^1_{-\alpha}$ and \label{a}
		\item  \hspace{.25in} $\Phi h (x_0) =0$.\label{b}
	\end{enumerate}
	
	To show (\ref{a}), we first show that $\Phi^\Gamma h_s \in v^{-1}H^2$. To see this, it suffices to argue that $\Phi^\Gamma h_s \perp \overline{H^2}$. To this end, let $\overline{g}\in \overline{H^2}$. Since $h_s \in \mathcal{S} = \spann\{ \overline{s_1},\ldots, \overline{s_d} \}$, it suffices to argue that $\langle \Phi^\Gamma \overline{s_j}, \overline{g} \rangle_2 = 0$ for all $1\leq j\leq d$. Recall, however, that the $s_j$ is simply a linear combination of reproducing functions at the points involved in the construction of $A$. Therefore, since $\Phi_\Gamma^\Gamma = (0,\ldots,0)^\top$,
	\[
	\langle \Phi^\Gamma \overline{s_j}, \overline{g} \rangle_2 = \int_{\partial X} \Phi^\Gamma \overline{s_j} g  \mathop{dm} = \int_{\partial X} (\Phi^\Gamma g) \overline{s_j}\mathop{dm} = \langle \Phi^{\Gamma} g, s_j \rangle_2 = 0.
	\]
	
	Thus we indeed have $\Phi^\Gamma h_s \in v^{-1}H^2$. Next we show that the product of $\Phi_v$ and $v^{-1}$ is bounded and analytic. By Proposition \ref{greenzeros}, we know $v$ has exactly $\tau$ many zeros. Thus, the $\tau$ zeros of $v$ will act as poles in $v^{-1}$ (including multiplicity), but will be canceled when multiplied by $\Phi_v$. Since there were exactly $\tau$ many poles, the product has no unbounded components and only a zero at $x_0$ (coming from the single pole at $x_0$ in $v$.) In this manner, $\Phi_v( v^{-1}H^2 ) \subseteq H^2$. In particular, since we showed that $\Phi^\Gamma h_s \in v^{-1}H^2$, it follows that
	\[
	\Phi h_s = \Phi_v( \Phi^{\Gamma} h_s ) \in \Phi_v( v^{-1}H^2 ) \subseteq H^2.
	\]
	Thus
	\[
	\Phi h = \Phi( v^{-1}r ) + \Phi h_s = \Phi^{\Gamma}( \Phi_v v^{-1}r ) + \Phi h_s \in H^1.
	\]
	
	Since $\Phi h \in H^1$, it is also in $H^1_{-\alpha}$. This shows (\ref{a})
	
	We noted that $\Phi^{\Gamma} h_s \in v^{-1}H^2$. Therefore there exists a function $b\in H^2$ such that $\Phi^\Gamma h_s = v^{-1}b$. Since $v^{-1}$ has a zero at $x_0$, $\Phi_v v^{-1}$ is an analytic function with a zero at $x_0$. Therefore,
	\begin{align*}
	\int_{\partial X} \Phi h \mathop{dm} &= \int_{\partial X} \Phi ( v^{-1}r) + \Phi h_s \mathop{dm}\\
	&= \int_{\partial X} \Phi^{\Gamma}r \Phi_v v^{-1} \mathop{dm} + \int_{\partial X } \Phi_v v^{-1}b \mathop{dm}\\
	&= (\Phi^\Gamma r\Phi_vv^{-1})(x_0) + (\Phi_v v^{-1}b)(x_0)\\
	&= 0
	\end{align*}
	This shows (\ref{b}).

	Since $\Phi h \in \mathcal{M} \subseteq H^1_{-\alpha}$, it follows from Lemma \ref{babyfactorization} that there exists $\omega \in \Sigma$ and $F,G \in H^2$ such that $\Phi h = FG$ and $\| \Phi h\|_{1,-\alpha} = \|F\|_{2,-(\alpha + \omega)}\|G\|_{2, -\alpha + \omega }$, $F$ is invertible, and $F \in \mathcal{M}$. In this manner, we find that
	\[
	\|h\|_1 = \|\Phi h\|_{1,-\alpha} = \|F\|_{2,-(\alpha - \omega)}\|G\|_{2, -\alpha + \omega }.
	\]
	There must exist some $D\in \Delta$ such that $F\in H^2_{ -(\alpha+\omega) , D}$. Since $F$ is invertible, $G$ will have to inherit the zero at $x_0$. 
	
	Now consider the function $\Phi^{-1}$. This function is not necessarily analytic, but we \textit{do} have that $|\Phi^{-1}| = |Z|^{-\alpha}$ \textit{on the boundary}. Form the function $g = \Phi^{-1}G$. Observe that, since $\Phi h = FG$, not only do we have
	\[
	h = \Phi^{-1} FG = Fg,
	\]
	but also
	\[
	\|g\|^2_{2,\alpha+\omega} = \int_{\partial X} |g|^2 |Z|^{\alpha+\omega}\mathop{dm} = \int_{\partial X} |G|^2 |\Phi^{-1}|^2 |Z|^{\alpha+\omega} \mathop{dm} = \int_{\partial X} |G|^2 |Z|^{-\alpha+\omega}\mathop{dm} = \|G\|^2_{2,-\alpha+\omega}
	\]
	so that we have the desired norm-factorization for $h$:
	\[
	\|h\|_1 = \| \Phi h \|_{1,-\alpha} = \|F\|_{2,-(\alpha - \omega)}\|G\|_{2, -\alpha + \omega } = \|F\|_{2,-(\alpha - \omega)}\|g\|_{2,\alpha+\omega}.
	\]
	Put $\beta :=  \alpha+\omega \in \Sigma$.  It remains to show that $\langle \psi,\overline{g} \rangle_2=0$ for all $\psi \in H^2_{ -\beta,D }$. Since $\Phi_\Gamma^\Gamma=(0,\ldots,0)^\top$, it follows that $\Phi H^2 = \Phi_v( \Phi^\Gamma H^2 )$ is a finite codimensional subspace of $H^2_{ -\beta,D }$. Thus, there must exist $\rho$ vectors $w_1,\ldots,w_\rho \in H^2_{ -\beta,D }$ such that
	\[
	H^2_{ -\beta,D } = \Phi H^2 \oplus \spann\{ w_1,\ldots,w_\rho  \}.
	\]

	Observe that if $f\in H^2$, then (since $G(x_0)=0$),
	\[
	\langle \Phi f, \overline{g}  \rangle_2 = \int_{\partial X} \Phi fg \mathop{dm} = \int_{\partial X} f \Phi \Phi^{-1} G   \mathop{dm} = \int_{\partial X} fG  \mathop{dm} = 0.
	\]
	Therefore, it suffices to show that $\langle w_j, \overline{g} \rangle_2 = 0$ for $1\leq j\leq \rho$. 
	
	To this end, recall that Proposition \ref{greenmeasure} gives us that $dm = \textstyle \frac{1}{2\pi i }vdz$. It follows that $v^{-1} dm = \textstyle \frac{1}{2\pi i} dz$. Moreover, it follows from the Cauchy integral theorem that, since $\partial X$ is a finite union of closed curves, $ \int_{ \partial X }f v^{-1}dm = \frac{1}{2\pi i} \oint f dz = 0$ for any analytic, $L^1$ function $f$. With these observations (along with the fact that $\Phi^{-1}h^{-1}G = F^{-1}$),
	\begin{equation} \label{equation1?}
	\langle w_j, \overline{g} \rangle_2 = \int_{\partial X} w_j \Phi^{-1}G  \mathop{dm} = \int_{\partial X} w_j  \Phi^{-1}Gh^{-1}h \mathop{dm} = \int_{\partial X} w_jF^{-1} (v^{-1}r + h_s)  \mathop{dm}.
	\end{equation}
	We endeavor to show that the right-most integral in (\ref{equation1?}) is equal to zero. Since $ w_jF^{-1}r$ is an analytic, $L^1$ function on $\partial X$, $\textstyle \int_{\partial X} (w_j F^{-1}r) v^{-1} \mathop{dm} =0$. Thus, it suffices to show that $\textstyle \int_{\partial X}  w_jF^{-1}  h_s \mathop{dm} = 0$ for all $w_j$. However, since $h_s \in \mathcal{S} = \spann\{ \overline{s_1},\ldots,\overline{s_d} \}$, it actually suffices to show that
	\[
	\int_{\partial X}  w_jF^{-1} \overline{s_i}  \mathop{dm} =0
	\]
	for any $i,j$. 
	
	Since $w_j, F \in H^2_{-\beta,D}$, it follows from Lemma \ref{landsinDgamma} that $w_j F^{-1} \in H^2_{D_\Gamma}$. That is, it satisfies the constraints of $A$. Recall that the $s_i$ are linear combination of reproducing functions at the various points involved in the construction of $A$. In particular, if any analytic function $\phi$ satisfies the constraints of $A$,  then $\langle \phi, s_i \rangle_2 = 0$ for all $i$. This implies that
	\[
	\int_{\partial X}  w_jF^{-1} \overline{s_i}  \mathop{dm} = \left\langle  w_jF^{-1} , s_i \right\rangle_2 = 0
	\]
	for all $i$ and $j$. Therefore $\langle \psi, \overline{g}\rangle_2 = 0$ for $\psi\in H^2_{-\beta,D}$. This completes the proof.
\end{proof}

\subsection{Universal Lower Bound for the Left-Invertible Toeplitz Operators} \label{lowerboundsection}

Let $\phi \in L^\infty$ be fixed and consider the Toeplitz operator $T^{\alpha,D}_\phi$. If we assume this operator is left-invertible, one can find $\varepsilon_{\alpha,D}>0$ such that $\|T^{\alpha,D}_\phi f\| \geq \varepsilon_{\alpha,D} \|f\|$ for all $f\in H^2_{\alpha,D}$. The goal of this subsection is to prove a uniform version of this statement:

\vspace{.1in}

\begin{proposition}\label{prop:univlowerboundtoep}
If $\phi \in L^\infty$ and $T^{\alpha,D}_\phi$ is left-invertible for every $(\alpha,D) \in \Sigma \times \Delta$, then there exists $0<\varepsilon<1$ (independent of $(\alpha,D)$) such that, for every $(\alpha,D) \in \Sigma \times \Delta$ and every $f\in H^2_{\alpha,D}$,
	\[
	\| T^{\alpha,D}_\phi f \| \geq \varepsilon \|f\|_\alpha
	\]
\end{proposition}
      
We first need a few lemmas. Let $L^2_0$ denote the usual, unweighted $L^2$ space and we equip $\Sigma$ with its usual metric.
\vspace{.1in}

\begin{lemma}\label{lem:inner-prod-continuous}
There is a universal constant $C$ such that for all $h_1, h_2\in L^2_0$ and all $\alpha, \beta\in \Sigma$
\[
|\langle h_1, h_2\rangle_\alpha -\langle h_1, h_2\rangle_\beta| \leq C\|h_1\|_0\|h_2\|_0 dist(\alpha, \beta). 
\]
\end{lemma}
\begin{proof}
  We first note that, since the functions $Z$ are uniformly bounded above and below on $X$, there is a uniform constant $C$ such that for all $\alpha, \beta\in \Sigma$,
\[
||Z|^{\alpha-\beta}-1| \leq C\cdot dist(\alpha, \beta)
\]
on $X$. 
  The proof is now straightforward: we have
\begin{align*}
\langle h_1, h_2\rangle_\alpha -\langle h_1, h_2\rangle_\beta &= \int h_1\overline{h_2} (|Z|^\alpha-|Z|^\beta)\, dm \\
&= \int h_1 \overline{h_2} (|Z|^{\alpha-\beta}-1)|Z|^\beta\, dm
\end{align*}
so 
\[
|\langle h_1, h_2\rangle_\alpha -\langle h_1, h_2\rangle_\beta| \leq C\cdot dist(\alpha, \beta) \int |h_1\overline{h_2}| |Z|^\beta\, dm \leq C\cdot dist(\alpha, \beta)\|h_1\|_\beta\|h_2\|_\beta
\]
and the conclusion follows by the uniform equivalence of the $\beta$ and $0$ norms. 
\end{proof}
	\begin{lemma}\label{lem:alpha-holder}
          Let $\mathcal H$ be a closed subspace of $L^2_0$. For each $\alpha\in \Sigma$, let $\mathcal H_\alpha$ be the image of the space $\mathcal H=\mathcal H_0$ under the identity mapping $\iota_\alpha:L_0^2\to L^2_\alpha$. let $P_\alpha$ denote the  orthogonal projection of $L^2_\alpha$ onto $\mathcal H_\alpha$. We let $Q_\alpha$ denote the corresponding operator in $L^2_0$:
          \[
            Q_\alpha := \iota_\alpha^{-1}P_\alpha \iota_\alpha.
          \]
          Then there exists an absolute constant $C$ such that for all $f\in L^2_0$ and all $\alpha, \beta\in \Sigma$, 
		\begin{equation*}
		\|Q_\alpha f-Q_\beta f\|_0^2 \leq C\cdot dist(\alpha, \beta)\|f\|_0^2.
		\end{equation*}
	\end{lemma}
	\begin{proof}
	Fix $f$. To unclutter the notation let $g_\alpha=P_\alpha f$ and $g_\beta=P_\beta f$. By definition, $g_\alpha$ is the unique vector in $H_\alpha$ such that
	\[
	\langle f-g_\alpha, h\rangle_\alpha = 0 \quad \text{for all }h\in H=H_\alpha
	\]
	and similarly for $g_\beta$. We then have, since the spaces $\mathcal H_\alpha$ all coincide with $\mathcal H$, 
	\begin{align*}
	\|g_\alpha-g_\beta\|_\alpha^2 &= \langle g_\alpha -g_\beta, g_\alpha-g_\beta\rangle_\alpha \\
	&=\langle f-g_\beta, g_\alpha-g_\beta\rangle_\alpha -\langle f-g_\alpha, g_\alpha-g_\beta\rangle_\alpha \\
	&= \langle f-g_\beta, g_\alpha-g_\beta\rangle_\alpha -\langle f-g_\beta, g_\alpha-g_\beta\rangle_\beta 
	\end{align*}
This last expression has the form $\langle h_1, h_2\rangle_\alpha -\langle h_1, h_2\rangle_\beta$, where we have put
\[
h_1 = f-g_\beta, \quad h_2 = g_\alpha-g_\beta.
\]
Observe that $\|g_\alpha\|_\alpha \leq \|f\|_\alpha\leq C\|f\|_0$, similarly for $\beta$, so that $\|h_1\|_0, \|h_2\|_0\leq C\|f\|_0$. Hence, by the continuity of the inner product (Lemma~\ref{lem:inner-prod-continuous}), 
\[
\|g_\alpha-g_\beta\|_\alpha^2 \leq |\langle h_1, h_2\rangle_\alpha-\langle h_1, h_2\rangle_\beta|\leq C\|h_1\|_0\|h_2\|_0 \cdot dist(\alpha, \beta) \leq C\|f\|_0^2\cdot dist(\alpha, \beta), 
\]
and the proof is finished by the mutual equivalence (with a universal constant) of the $\alpha$ and $0$ norms.  
	\end{proof}
	We observe that the constants in the above argument do not depend on the choice of the original subspace $\mathcal H$. It then follows immediately that 
	\begin{lemma}
	There is a universal constant $C$ such that for all $f\in L^2_0$ and all $D\in \Delta$, 
	\[
	\|Q_{\alpha, D}f-Q_{\beta, D}f\|_0^2\leq C\cdot dist(\alpha, \beta)\|f\|_0^2.
	\]
	In particular, for fixed $D\in \Delta$, the map from $\Sigma$ to $B(L^2_0)$ given by
	\[
	\alpha\to Q_{\alpha, D}:= \iota_\alpha^{-1} P_{\alpha,D} \iota_\alpha
	\]
	is H\"older continuous, with constant independent of $D\in \Delta$:
	\[
	\|Q_{\alpha, D}-Q_{\beta, D}\|_{B(L^2_0)} \leq C\cdot dist(\alpha, \beta)^{1/2}.
	\]
	\end{lemma}
	
	\begin{lemma}\label{lem:delta-lipschitz}
	For fixed $\alpha\in \Sigma$, the map from $\Delta$ to $B(L^2_\alpha)$ given by
\[
	D\to P_{\alpha, D}
\]
is Lipschitz continuous, with constant independent of $\alpha$, that is, for $D, D^\prime\in \Delta$
\[
\|P_{\alpha, D}-P_{\alpha, D^\prime}\|_{B(L^2_\alpha)} \leq C\cdot dist (D, D^\prime).
\]	

	\end{lemma}
	\begin{proof}
          We give the proof in the case of a single 2-point constraint; the case of a single derivation is similar. The general case then follows by a straightforward induction on the codimension, we leave the details to the reader.

          Suppose we have a 2-point constraint $f(a)=f(b)$. In this case $\Delta$ is a copy of the Riemann sphere, and for fixed $D\in \Delta$ there exist complex numbers $t_a, t_b$ such that $|t_a|^2+|t_b|^2=1$, and such that the subspace $H^2_{\alpha, D}$ consists of those functions $f\in H^2_\alpha$ for which
          \[
            t_af(a)+t_bf(b)=0.
          \]
          The $t_a, t_b$ are uniquely determined if we impose the additional requirement that $t_a\geq 0$, which we do from now on. The space $\Delta$ is then a metric space if we impose, say, the $\ell^1$ metric. Since $I-P_{\alpha, D}$ is the rank-one projection onto the difference of reproducing kernels $t_ak_a^\alpha+t_bk_b^\alpha$, to prove the desired continuity it suffices to observe that for fixed $a$ and $b$, the norms of the reproducing kernels $k_a^\alpha, k_b^\alpha$ are uniformly bounded above and below (away from zero), independently of $\alpha$. 
	\end{proof}

	\begin{proposition}\label{prop:norm-continuity}
	The map from $\Sigma\times \Delta$ to $B(L^2_0)$ given by 
	\[
	(\alpha, \Delta) \to Q_{\alpha, D} := \iota_\alpha^{-1} P_{\alpha, D}\iota_\alpha
	\]
	is continuous. 
	\end{proposition}
	\begin{proof}
	This follows immediately from Lemmas~\ref{lem:alpha-holder} and \ref{lem:delta-lipschitz}, and the fact that the $L^2_\alpha$ norms are all mutually equivalent, with uniform constants. 
	\end{proof}

\begin{proof}[Proof of Proposition~\ref{prop:univlowerboundtoep}] For $(\alpha, D)\in \Sigma\times \Delta$, we define operators $Q_{\alpha, D}$ and $X_{\alpha, D}$ in $L^2_0$ by
\[
Q_{\alpha, D} := \iota_\alpha^{-1} P_{\alpha, D}\iota_\alpha
\]
and
\[
X_{\alpha, D} := Q_{\alpha, D}M_\phi Q_{\alpha, D} +(I-Q_{\alpha, D}).
\]
By Proposition~(\ref{prop:norm-continuity}), we have that the map $(\alpha, D)\to  X_{\alpha, D}$ is norm continuous. 

For fixed $(\alpha, D)$, there is by hypothesis an $\epsilon(\alpha, D)\in (0, 1]$ such that $\|V_{\alpha, D}T_\phi^{\alpha, D}f\| =\|T_\phi^{\alpha, D}f\| \geq \epsilon(\alpha, D) \|f\|$ for $f\in H^2_\alpha$. Thus, given $F\in L^2_0$ and decomposing it as $F=f+g$ with $f\in H^2_{\alpha, D}$ and $g\in (H^2_{\alpha, D})^\bot$, (the orthogonal complement taken in $L^2_\alpha$), we have
\begin{align*}
\|X_{\alpha, D}F\|^2 &\geq C\|\iota_\alpha^{-1}X_{\alpha, D}F\|^2_{\alpha} \\
&=C(\|V_\alpha T_\phi^{\alpha, D}f\|^2_{\alpha} +\|g\|_{\alpha}^2 )\\
&\geq C\epsilon(\alpha, D)^2(\|f\|^2_\alpha +\|g\|^2_\alpha)\\
&\geq C^\prime\epsilon(\alpha, D)^2\|F\|_0^2.
\end{align*}
Absorbing the constant $C^\prime$ into the definition of $\epsilon$, it follows that there exists $\epsilon(\alpha, D)>0$ such that $\|X_{\alpha, D}F\|\geq \epsilon(\alpha, D)\|F\|$ for all $F\in L^2_0$. 

To show that this holds with a {\em uniform} choice of $\epsilon>0$, suppose no such uniform choice exists; then there exists a sequence $(\alpha_n, D_n)$ from $\Sigma\times \Delta$ and unit vectors $F_n\in L^2_0$ such that $\|X_{\alpha_n, D_n}F_n\|\to 0$. By compactness, we may assume $(\alpha_n, D_n)$ converges to some $(\alpha, D)$. Then
\[
0<\epsilon(\alpha, D) \leq \|X_{\alpha, D} F_n\|\leq \|X_{\alpha_n, D_n}F\|+\|(X_{\alpha, D}-X_{\alpha_n, D_n})F\|
\]
By norm continuity of $X_{\alpha, D}$, the right hand side tends to $0$ as $n$ tends to infinity, which is a contradiction. 
\end{proof}

\subsection{Invertibility for Toeplitz Operators and Symbols}

In this section we collect a few necessary lemmas on the invertibility of not only the Toeplitz operators themselves, but their relation to the invertibility of their symbols as functions in the algebra $A$.

\begin{lemma} \label{multiplicationtoeplitzsymbol}
	If $\phi \in L^\infty$ and $\psi \in A$, then for all $(\alpha,D)\in \Sigma \times \Delta$,
	\[
	T^{\alpha,D}_{ \overline{\psi}\phi } = T^{\alpha,D}_{ \overline{\psi} }T^{\alpha,D}_{\phi} \hspace{.25in} \text{and} \hspace{.25in} T^{\alpha,D}_{ \psi \overline{\phi} } = T^{\alpha,D}_{ \overline{\phi} }T^{\alpha,D}_{ \psi }.
	\]
\end{lemma}
\begin{proof} Let $f,g \in H^2_{\alpha,D}$. Since $A$ acts as a multiplier algebra for $H^2_{\alpha,D}$, we also have $\psi g \in H^2_{\alpha,D}$. It follows from Lemma \ref{adjointoftoeplitz} that
	\begin{align} \label{multiplicationtoeplitzsymbol1}
	\langle T^{\alpha,D}_{\overline{\psi} } T^{\alpha,D}_{\phi} f, g \rangle_\alpha = \langle T^{\alpha,D}_\phi f, T^{\alpha,D}_\psi g \rangle_\alpha = \langle V_{\alpha,D}^* M_\phi V_{\alpha,D} f, V_{\alpha,D}^* M_\psi V_{\alpha,D}g \rangle_\alpha.
	\end{align}
	Now, since $\psi g \in H^2_{\alpha,D}$, we have that $V_{\alpha,D}^* M_\psi V_{\alpha,D}g = \psi g$. However, since $\phi f$ may not be in $H^2_{\alpha,D}$, we have $V_{\alpha,D}^* M_\phi V_{\alpha,D}f = V_{\alpha,D}^* \phi f$. These observations allow us to see that
	\begin{align} \label{multiplicationtoeplitzsymbol2}
	\langle V_{\alpha,D}^* M_\phi V_{\alpha,D} f, V_{\alpha,D}^* M_\psi V_{\alpha,D}g \rangle_\alpha = \langle V_{\alpha,D}^* \phi f, \psi g \rangle_\alpha = \langle \phi f, V_{\alpha,D}\psi g\rangle_\alpha = \langle \overline{\psi} \phi f, g \rangle_\alpha.
	\end{align}
	Combining (\ref{multiplicationtoeplitzsymbol1}) and (\ref{multiplicationtoeplitzsymbol2}), we have
	\[
	\langle T^{\alpha,D}_{\overline{\psi} } T^{\alpha,D}_{\phi} f, g \rangle_\alpha = \langle \overline{\psi} \phi f, g \rangle_\alpha = \langle V_{\alpha,D} \overline{\psi} \phi f, V_{\alpha,D} g \rangle_\alpha = \langle \overline{\psi} \phi V_{\alpha,D}f, V_{\alpha,D} g \rangle_\alpha = \langle V_{\alpha,D}^* \overline{\psi} \phi V_{\alpha,D} f, g \rangle_\alpha.
	\]
	But, by definition $V_{\alpha,D}^* \overline{\psi} \phi V_{\alpha,D} = T^{\alpha,D}_{ \overline{\psi}\phi }$, so the above becomes:
	\[
	\langle T^{\alpha,D}_{\overline{\psi} } T^{\alpha,D}_{\phi} f, g \rangle_\alpha = \langle T^{\alpha,D}_{ \overline{\psi}\phi }f,g\rangle_\alpha.
	\]
	Since this holds for all $f,g \in H^2_{\alpha,D}$, it follows that $T^{\alpha,D}_{ \overline{\psi}\phi } = T^{\alpha,D}_{\overline{\psi} } T^{\alpha,D}_{\phi}$. By taking adjoints (and therefore another application of Lemma \ref{adjointoftoeplitz}), we conclude
	\[
	T^{\alpha,D}_{ \psi \overline{\phi} } = ( T^{\alpha,D}_{ \overline{\psi}\phi } )^* = ( T^{\alpha,D}_{\overline{\psi} } T^{\alpha,D}_{\phi}   )^* = ( T^{\alpha,D}_\phi  )^*( T^{\alpha,D}_{ \overline{\psi } }  )^* = T^{\alpha,D}_{ \overline{\phi} }T^{\alpha,D}_{\psi}.
	\]
\end{proof}

Declare an element $\phi \in A$ to be \textit{invertible} in $A$ if $\phi(z) \neq 0$ for all $z\in X$ and $\phi^{-1} = \textstyle \frac{1}{\phi} \in A$.

\begin{lemma} \label{TFAE}
	Let $\psi \in A$. The following are equivalent:
	\begin{enumerate}[(i)]
		\item   $\psi$ is invertible in $A$;
		\item   There exists $(\alpha,D)\in \Sigma \times \Delta$ such that $T^{\alpha,D}_\psi$ is right invertible;
		\item  $T^{\alpha,D}_\psi$ is invertible for every $(\alpha,D)\in \Sigma \times \Delta$.
	\end{enumerate}
	Moreover, $(T^{\alpha,D}_\psi)^{-1} = T^{\alpha,D}_{ \psi^{-1} }$.
\end{lemma}
\begin{proof}
	To begin, we suppose $\psi$ is invertible in $A$ and let $(\alpha,D)\in \Sigma \times \Delta$. By definition, $\psi$ does not vanish on $X$ and $\psi^{-1}\in A$. In particular, $\psi^{-1}\psi = \psi \psi^{-1} =1$. This implies that
	\[
	T^{\alpha,D}_{ \psi \psi^{-1} } = T^{\alpha,D}_1 = I = T^{\alpha,D}_{ \psi^{-1}\psi }.
	\]
	Where $I$ is the identity operator. Now, we observe that both $\overline{\psi}$ and $\overline{ \psi^{-1} }$ are in $L^\infty$. Therefore, by applying Lemma \ref{multiplicationtoeplitzsymbol}, we find that
	\[
	T^{\alpha,D}_{ \psi^{-1} } T^{\alpha,D}_\psi = T^{\alpha,D}_{ \psi \psi^{-1} } = I = T^{\alpha,D}_{ \psi^{-1}\psi } = T^{\alpha,D}_{\psi} T^{\alpha,D}_{ \psi^{-1} }.
	\]
	This shows that $T^{\alpha,D}_\psi$ is invertible for every $(\alpha,D)\in \Sigma \times \Delta$. Moreover, we see that its inverse is $T^{\alpha,D}_{ \psi^{-1} }$. This establishes that $(i)$ implies $(iii)$. It is clear that $(iii)$ implies $(ii)$. Therefore it remains to show that $(ii)$ implies $(i)$. 
	
	To this end, let $(\alpha,D) \in \Sigma \times \Delta$ be a parameter such that $T^{\alpha,D}_\psi$ is right invertible. Then there exists an operator $X^{\alpha,D}_\psi$ such that $T^{\alpha,D}X^{\alpha,D}_\psi = I$ and so, by taking adjoints, $(X^{\alpha,D}_\psi)^* (T^{\alpha,D}_\psi)^* = I$. It therefore follows that, for any $f\in H^2_{\alpha,D}$,
	\[
	\|f\| = \| (X^{\alpha,D}_\psi)^* ( T^{\alpha,D}_\psi)^* f\| \leq \| (X^{\alpha,D}_\psi)^* \| \| (T^{\alpha,D}_\psi)^* f\|.
	\]
	Thus, if we put $\delta := \frac{1}{ \| (X^{\alpha,D}_\psi)^* \|  } >0$, we have that $\| (T^{\alpha,D}_\psi)^* f \| \geq \delta \|f\|$. Let $k_w^{\alpha,D}(z)$ be the reproducing kernel for $H^2_{\alpha,D}$. Observe that if $f\in H^2_{\alpha,D}$, then
	\[
	\langle f, (T^{\alpha,D}_\psi)^* k_w^{\alpha,D} \rangle_\alpha = \langle T^{\alpha,D}_\psi f, k_w^{\alpha,D}\rangle_\alpha = \langle \psi f, k_w^{\alpha,D}\rangle_\alpha = \psi(w) \langle f, k_w^{\alpha,D} \rangle_\alpha = \langle f, \overline{\psi(w)} k_w^{\alpha,D} \rangle_\alpha.
	\]
	Thus we yield the following eigenvector relationship:
	\[
	(T^{\alpha,D}_\psi)^* k_w^{\alpha,D} = \overline{\psi(w) } k_w^{\alpha,D}.
	\]
	Since $k_w^{\alpha,D}\in H^2_{\alpha,D}$, we have that
	\begin{equation} \label{kernelcancel}
	| \overline{\psi(w)} | \| k_w^{\alpha,D} \| = \|  (T^{\alpha,D}_\psi)^* k_w^{\alpha,D} \| \geq \delta \|k_w^{\alpha,D}\|. 
	\end{equation}
	Let $\Omega = \{ w\in X \ : \ \psi(w)\neq0 \ \text{and} \ k_w^{\alpha,D}\neq0 \}$. Observe that $\Omega$ is a dense subset of $X$. Moreover, on $\Omega$, we can form $\textstyle \frac{1}{\psi}$ and divide by $\|k_w^{\alpha,D}\|$. Thus, (\ref{kernelcancel}) implies that $| \textstyle \frac{1}{\psi}(w)| \leq \frac{1}{\delta}$ for $w\in \Omega$ and hence everywhere by continuity. Thus $\psi$ does not vanish on $X$ and therefore $\psi$ is invertible in $A$.
\end{proof}

\section{The Widom theorem for constrained algebras}\label{sec:widom}
\begin{lemma} \label{distanceimpliesinvertibility}
	Let $\phi \in L^\infty$ be unimodular. If there exists $\psi \in A$ such that $\|\phi - \psi\| <1$, then $T^{\alpha,D}_\phi$ is left-invertible for every $(\alpha,D)\in \Sigma \times \Delta$. Further, if $\psi$ is invertible in $A$, then $T^{\alpha,D}_\phi$ is invertible for every $(\alpha,D)\in \Sigma \times \Delta$.
\end{lemma}
\begin{proof}
	Suppose $\psi\in A$ is such that $\|\phi - \psi\| <1$. Since $\phi$ is unimodular, $\|1-\psi \overline{\phi}\| \leq1$.
	Therefore, it follows from Lemma \ref{adjointoftoeplitz}, that 
	\[
	\|1- T^{\alpha,D}_{\phi\overline{\phi} } \| = \| T^{\alpha,D}_{ 1-\psi \overline{\phi} } \| = \| 1-\psi \overline{\phi} \| <1.
	\]
	This directly implies that $T^{\alpha,D}_{ \psi \overline{\phi} }$ is invertible. Thus, by Lemma \ref{multiplicationtoeplitzsymbol}, it follows that $T^{\alpha,D}_\phi$ is left-invertible. Suppose further that $\psi$ is invertible in $A$. By Lemma \ref{TFAE}, we have that $T^{\alpha,D}_\psi$ is invertible. It follows that $T^{\alpha,D}_\phi$ is right invertible and therefore $T^{\alpha,D}_\phi$ is totally invertible.
\end{proof}

\begin{lemma} \label{CSforalphaspaces}
	If $f,g \in L^2$, then for all $\alpha \in \Sigma$,
	\[
	| \langle f,g \rangle_2 | \leq \|f\|_{2,-\alpha}\|g\|_{2,\alpha}.
	\]
\end{lemma}
\begin{proof}
	Let $\alpha \in \Sigma$. By the classic Cauchy-Schwarz inequality,
	\[
	|\langle f,g \rangle_2 | = \left| \int_{ \partial X} f \overline{g} \mathop{dm} \right| = \left| \int_{ \partial X} f \overline{g} |Z|^{-\frac{\alpha}{2} }|Z|^{ \frac{\alpha}{2} } \mathop{dm} \right|= \left| \langle f |Z|^{-\frac{\alpha}{2}} , g \overline{|Z|^{ \frac{\alpha}{2}}} \rangle_2 \right|\leq \| f |Z|^{-\frac{\alpha}{2}} \|_2 \| g \overline{|Z|^{ \frac{\alpha}{2}}}\|_2
	\]
	Therefore
	\[
	|\langle f,g \rangle_2 | \leq \int_{\partial X } |f|^2 |Z|^{-\alpha} \mathop{dm} \int_{\partial X } |g|^2 |Z|^\alpha \mathop{dm}= \|f\|_{2,-\alpha} \|g\|_{2,\alpha}.
	\]
\end{proof}

\begin{lemma} \label{distiffinvert}
	Suppose $\phi \in L^\infty$ is unimodular. The distance from $\phi$ to $A$ is strictly less than one if and only if $T^{\alpha,D}_\phi$ is left-invertible for every $(\alpha,D)\in \Sigma \times \Delta$.
\end{lemma}
\begin{proof}
	Suppose that the distance from $\phi$ to $A$ is strictly less than one. Lemma \ref{distanceimpliesinvertibility} implies that $T^{\alpha,D}_\phi$ is left-invertible for every $(\alpha,D)\in \Sigma \times \Delta$. Now assume that $T^{\alpha,D}_\phi$ is left-invertible for every $(\alpha,D)\in \Sigma \times \Delta$.
	
	By Proposition \ref{prop:univlowerboundtoep}, there exists a uniform $\varepsilon>0$ (independent of $(\alpha,D)$) such that, for all $(\alpha,D)$ and $f\in H^2_{\alpha,D}$,
	\begin{equation} \label{distiffinvert1}
	\|T^{\alpha,D}_\phi f\| \geq \varepsilon \|f\|_\alpha.
	\end{equation}
	Let $h \in \mathscr{M}_\mathcal{M} = v^{-1}\mathcal{M} + \mathcal{S}$. By Lemma \ref{factorizationforMp}, there exists $(-\beta,D)\in \Sigma \times \Delta$, $f\in H^2_{-\beta,D}$, and $g \in L^2_{-\beta}$ such that $h=fg$, $\|h\|_1 = \|f\|_{2,-\beta}\|g\|_{2,\beta}$, and $\langle \psi, \overline{g} \rangle_2 =0$ for all $\psi \in H^2_{-\beta,D}$. Recall that $P_{-\beta,D}$ is the orthogonal projection from $L^2_{-\beta}$ onto $H^2_{-\beta,D}$. Thus, on one hand Lemma \ref{CSforalphaspaces} gives rise to the following estimate:
	\begin{align*}
	\left| \int_{\partial X} \phi h \mathop{dm} \right| &= \left|  \int_{\partial X} \phi fg \mathop{dm} \right|\\
	&= | \langle \phi f, \overline{g} \rangle_2|\\
	&= | \langle \phi f, (I-P_{-\beta,D}) \overline{g} \rangle_2|\\
	&= | \langle (I-P_{-\beta,D} )\phi f, \overline{g}\rangle_2|\\
	&\leq \|  (I-P_{ -\beta,D }) \phi f \|_{2, -\beta} \|g\|_{2,\beta}. 
	\end{align*}
	While on the other hand, the Pythagorean theorem and the fact that $\phi$ is unimodular asserts that
	\[
	\|f\|^2_{2,-\beta} = \|\phi f \|^2_{2,-\beta} = \|P_{-\beta,D} \phi f + (I-P_{-\beta,D}) \phi f \|^2_{2,-\beta} = \|P_{-\beta,D} \phi f \|^2_{2,-\beta} + \|(I-P_{-\beta,D}) \phi f \|^2_{2,-\beta}.
	\]
	Now, combining (\ref{distiffinvert1}) and the fact that $P_{-\beta,D} \phi f = T^{-\beta,D}_\phi f$, the above equality becomes
	\[
	\|f\|^2_{2,-\beta} = \|T^{-\beta,D}_\phi \phi f \|^2_{2,-\beta} + \|(I-P_{-\beta,D}) \phi f \|^2_{2,-\beta} \geq \varepsilon^2 \|f\|^2_{2,-\beta} + \|(I-P_{-\beta,D}) \phi f \|^2_{2,-\beta}.
	\]\
	and therefore
	\[
	\|(I-P_{-\beta,D}) \phi f \|^2_{2,-\beta} \leq \sqrt{1-\varepsilon^2  }\|f\|_{2,-\beta}.
	\]
	The above estimate, along with the estimate we had on the integral of $\phi h$ guarantee that
	\begin{equation} \label{distiffinvert2}
	\left|  \int_{\partial X} \phi h \mathop{dm}\right| \leq  \|  (I-P_{ -\beta,D }) \phi f \|_{2, -\beta} \|g\|_{2,\beta} \leq \sqrt{1-\varepsilon^2  }\|f\|_{2,-\beta}\|g\|_{2,\beta} = \sqrt{1-\varepsilon^2  }\|h\|_1.
	\end{equation}
	Note that the above estimate is holding for $h\in \mathscr{M}_\mathcal{M}$, a dense subset of $\mathscr{M}$. Since integrating against $\phi$ and $\|\cdot\|_1$ are each continuous linear functionals, the fact that inequality in (\ref{distiffinvert2}) holds for a dense subset of $\mathscr{M}$ immediately implies that it will hold for all $h\in \mathscr{M}$.
	Recall that by Lemma \ref{quotientisomorphictoM}, the map $\Lambda\colon L^\infty /A \to \mathscr{M}^*$ sending $\pi(\phi)\mapsto (\lambda_\phi)\big|_{\mathscr{M} }$, where $\lambda_\phi\colon L^1\to \C$ is the functional sending $\psi \mapsto \textstyle \int_{\partial X} \phi \psi$, is an isometric isomorphism. 
	
	Since (\ref{distiffinvert2}) holds for $h \in \mathscr{M}$, we find that $\|\lambda_\phi(h) \big|_{\mathscr{M}} \| <1$. Since $\Lambda$ is an isometry, this implies that $\|\pi(\phi)\|<1$. Since the norm of a vector is interpreted as its distance from the `zero' element, this implies that the distance from $\phi$ to $A$ is strictly less than one.
\end{proof}

\noindent\textbf{Theorem \ref{widomforA}} (Widom Theorem for $A$\textbf{.})
\textit{Suppose $\phi \in L^\infty$ is unimodular. $T^{\alpha,D}_\phi$ is left-invertible for each $(\alpha,D) \in \Sigma \times \Delta$ if and only if $\text{dist}(\phi, A) <1$. In particular, $T^{\alpha,D}_\phi$ is invertible for each $(\alpha,D)\in \Sigma \times \Delta$ if and only if $\text{dist}(\phi,A^{-1})<1$.}
\begin{proof}
	Everything aside from the `in particular' statement has been proven in Lemmas \ref{distanceimpliesinvertibility} and \ref{distiffinvert}.
	
	Thus, suppose $T^{\alpha,D}_\phi$ is invertible for reach $(\alpha,D)\in \Sigma\times \Delta$. Lemma \ref{distiffinvert} guarantees that there exists $\psi\in A$ such that $\|\psi-\phi\|<1$. We need to argue that $\psi$ is invertible in $A$.
	
	By Lemma \ref{distanceimpliesinvertibility}, we know that $T^{\alpha,D}_{ \overline{\phi} } T^{\alpha,D}_\psi$ and $T^{\alpha,D}_\phi$ are invertible. By taking its adjoint, we  have that $(T^{\alpha,D}_\phi)^* = T^{\alpha,D}_{ \overline{\phi} }$ is invertible. Since both $T^{\alpha,D}_{ \overline{\phi} } T^{\alpha,D}_\psi$ and $T^{\alpha,D}_{ \overline{\phi} }$ are invertible, we find that $T^{\alpha,D}_\psi$ is necessarily invertible. Applying Lemma \ref{TFAE}, we conclude that $\psi$ is invertible in $A$ as desired.
	
	Conversely, suppose there exists an invertible $\psi \in A$ such that $\|\psi - \phi\|<1$. Lemma \ref{distanceimpliesinvertibility} asserts that $T^{\alpha,D}_\phi$ is invertible for all $(\alpha,D)\in \Sigma \times \Delta$.
\end{proof}

\nocite{Abrahamse}

\bibliographystyle{plain}

\end{document}